\documentclass{article}
\usepackage{graphicx} 
\usepackage{amsfonts,amsthm, amsmath,amssymb,latexsym,psfrag,natbib}
\usepackage{eucal}
\usepackage{dsfont}
\usepackage{svg}
\usepackage{amssymb,epsfig,latexsym}
\usepackage[colorlinks=true, urlcolor=blue,linkcolor=blue, citecolor=blue]{hyperref}
\usepackage{algorithm}
\usepackage{algorithmic}
\usepackage{physics}
\usepackage{subcaption}
\usepackage[a4paper,
            bindingoffset=0.2in,
            left=1in,
            right=1in,
            top=1in,
            bottom=1in,
            footskip=.25in]{geometry}

\usepackage{cancel}

\newcommand{\R}{{\mathbb R}}

\newcommand{\N}{{\mathbb N}}

\newcommand{\Z}{{\mathbb Z}}

\newcommand{\PP}{{\mathbb P}}
\newcommand{\E}{{\mathbb E}}
\newcommand{\Pb}{{\mathbf P}}
\newcommand{\Eb}{{\mathbf E}}

\newcommand{\balambda}{{\bar \lambda}}

\newcommand{\btheta}{{\boldsymbol \theta}}

\newcommand{\bmu}{{\boldsymbol \mu}}

\newtheorem{Theorem}{Theorem}[section]
\newtheorem{Proposition}[Theorem]{Proposition}

\newtheorem{Lemma}[Theorem]{Lemma}
\newtheorem{Remark}[Theorem]{Remark}

\newtheorem{Assumption}[Theorem]{Assumption}

\usepackage[textwidth=1.8cm, textsize=scriptsize]{todonotes}

\title{Poisson Autoregression on a Large Network with a Stochastic Block Model Structure}
\author{Mahmoud Khabou$^1$, Joanna Marks$^1$, Joshua Corneck$^1$ and Edward A. K. Cohen$^1$}
\date{%
    $^1$Department of Mathematics, Imperial College London, South Kensington Campus, London, SW7 2AZ, U.K.\\%
}

\begin{document}

\maketitle
\begin{abstract}
  We consider multivariate Poisson autoregressive processes evolving on random networks generated by a Stochastic Block Model. This framework extends existing network Poisson autoregressions by incorporating latent community structure and allowing for non-Markovian dependence, both excitatory and inhibitory. We establish a mean-field approximation and prove that the finite system converges to its limit at the optimal rate $N^{-1/2}$ in the quenched setting with overwhelming probability, both on average and in uniform norm. Our analysis yields quantitative error bounds that are explicit in the time horizon and, under suitable stability conditions, remain polynomial or uniform in time, thereby avoiding the exponential growth typically associated with Grönwall-type arguments. As a consequence, the mean-field approximation is shown to hold uniformly over any time horizon. We further demonstrate how the approximation can be exploited to reduce the computational complexity of parametric inference from $\mathcal{O}(N^2T^2)$ to $\mathcal{O}(NT^2)$. The results provide a theoretical foundation for the analysis and inference of count-valued processes on large random networks with community structure.
\end{abstract}

\section{Introduction}
We consider observing a count series, that is, a time series taking non-negative integer values, on a system equipped with a network structure. Such series arise naturally in a wide range of applications, including spike trains in neuronal networks \citep{bio1,bio2}, the number of cases during an epidemic \citep{hung2025bayesian,kaur2024dynamic,khabou2025markov,count_med}, social network interactions \citep{social_1, social_2, group_network}, and measures of financial activity \citep{heinen2007multivariate,finance}. These diverse applications have motivated a growing body of research on stochastic processes evolving over networked systems, with particular emphasis on capturing both temporal dependence and interaction structures between nodes.

In continuous time, a substantial literature has developed around Hawkes processes and their behaviour on large networks. Early work by \cite{delattre_hawkes_2016} provides a first study of Hawkes processes on a deterministic fully connected graph, establishing both a law of large numbers and a central limit theorem, with the characteristic mean-field convergence rate of $N^{-1/2}$, where $N$ is the number of nodes. This result was subsequently generalised in \cite{pfaffelhuber_mean-field_2022}, where edges carry random labels determining whether interactions are excitatory or inhibitory, although no explicit convergence rate is obtained. Extensions to more general random graph settings, including inhomogeneous or sparse structures, are considered in \cite{agathe-nerine_hawkes}, where the same $N^{-1/2}$ rate is recovered in the annealed setting, averaging over graph realisations. Additional developments include mean-field approximations for spatially extended Hawkes processes in the Markovian case \citep{chevallier_duarte}, as well as variants incorporating age dependence \citep{chevallier-age_hawkes}, diffusive regimes \citep{erny}, and multiplicative inhibition \citep{duval_interacting_2021}. 

Despite this breadth of results in continuous time, comparatively less attention has been devoted to discrete-time counterparts.
Poisson autoregressive models, which may be viewed as discrete-time analogues of Hawkes processes \citep{coutin2024functionalapproximationmarkedhawkes}, have been studied to a much lesser extent, despite their particular relevance in light of the widespread availability of integer-valued data. For a general overview of multivariate count time series, including Markov Poisson autoregressions within the class of observation-driven models, we refer to \cite{fokianos_multi}, while foundational aspects of such time series are discussed in \cite{doukhan}. When multivariate data are associated with an underlying network structure, network Poisson autoregressions become especially pertinent. In this direction, \cite{AF1} and \cite{AF2} investigate Markov Poisson autoregressions defined on a known deterministic graph, establishing stability properties and consistency of the quasi-maximum likelihood estimator, including in regimes where the number of nodes grows large.

A common limitation of the models described above is their assumption of a fixed, deterministic graph structure. In many applications, however, the network itself exhibits latent community structure: nodes belong to unobserved groups, and edge formation depends on these group memberships. Depending on the context, such groups may reflect underlying spatial organization, socioeconomic or demographic similarity, or functional roles played by the nodes. These forms of latent structure have been documented in a range of domains, including neuronal networks \citep{sbm_brain}, transport networks \citep{corneck25a} and political networks \citep{sbm_politics}. The Stochastic Block Model (SBM) provides a natural framework for capturing this type of heterogeneity, by specifying edge probabilities through latent block assignments. This, in turn, motivates the study of stochastic processes evolving on random graphs with community structure.

In this paper, we consider a multivariate Poisson autoregression on a random graph generated by an SBM. Such a framework can represent, for instance, the occurrence times of neuronal action potentials in a network with excitatory and inhibitory interactions \citep{Biology1,bio1,bio2}. In contrast to existing network Poisson autoregressive models \citep{AF1,AF2}, our framework allows for random interactions induced by the SBM and is not restricted to the Markovian setting, thereby accommodating longer memory effects.

The contribution of this paper is fourfold. First, we establish a mean-field approximation for Poisson autoregressions on an SBM graph. Second, we prove that the canonical $N^{-1/2}$ convergence rate holds in the quenched setting with overwhelming probability, both on average over nodes and in the uniform norm. Third, we obtain explicit bounds in the time horizon that are polynomial or even uniform in $T$, thereby significantly improving upon the classical $\mathrm e^{T}$ bounds derived via Grönwall-type arguments. As a consequence, the mean-field approximation is shown to hold uniformly for all times $t \in \mathbb{N}$ under suitable stability conditions. Finally, we demonstrate how this approximation can be leveraged to develop a simplified approach to parametric inference that reduces the typical $\mathcal{O}(N^2T^2)$ complexity to $\mathcal{O}(NT^2)$.

The remainder of the paper is organized as follows. General definitions regarding matrices and Poisson processes are provided in Section \ref{sec:def_general}. Section \ref{sec: def_sbm} introduces the Stochastic Block Model and states a key approximation result. The Poisson autoregressive model is defined in Section \ref{sec: def_pois}, where stability properties are also discussed. In Section \ref{sec:heuristics}, we provide heuristic insights into the behaviour of the system as the number of nodes grows large. Rigorous convergence results are presented in Section \ref{sec:convergence-N}, with convergence in average treated in Section \ref{sec:cvg_euclidean} and uniform convergence in Section \ref{sec:cvg_uniform}. Finally, Section \ref{sec:analogy} establishes analogous results for Hawkes processes on SBM graphs. Technical lemmas are collected in Appendix \ref{App:lemmas}, while proofs are deferred to Appendix \ref{App:proofs}.

\section{Definitions and preliminary results}
\subsection{General definitions and notations}
\label{sec:def_general}

Let $N\in \N^*$ be the total number of nodes in a given network, \textit{e.g.} the number of neurons in a part of the brain, the number of users on a social network or the number of financial agents. For a vector $x \in \R^N$, $|x|$ denotes $\left(|x^1|,\cdots,|x^N|\right)$. We say that $x$ is non-negative if $x=|x|$. For two non-negative vectors $x$ and $y$, we write $x \preceq y$ if $x^i \leq y^i$ for all $i\in \{1,\cdots,N\}$. The same definitions can be naturally extended to matrices of the same shape.

We denote by $\|\cdot\|_2$ and $\|\cdot\|_\infty$ the Euclidean and uniform norms respectively, that is, $\|x\|_2=(\sum_{i=1}^N |x^i|^2)^{1/2}$ and $\|x\|_\infty=\max_{i\in \{1,\cdots,N\}}|x^i|$.
 We also denote by $\|\cdot\|_2$ the matrix norm induced by the Euclidean norm, that is 
 $\|A\|_2=\sup_{x\neq 0}(\|Ax\|_2/\|x\|_2).$ This norm is also referred to as the spectral norm and it is known that $\|A\|_2=[\rho(A^*A)]^{1/2}$, where $A^*$ is the adjoint of $A$ and $\rho$ denotes the spectral radius, that is the largest absolute eigenvalue. Note that if $A$ is real symmetric, then $\|A\|_2=\rho(A)$.
 Similarly, $\|\cdot\|_\infty$ denotes the matrix norm induced by the uniform norm, that is, $\|A\|_\infty=\sup_{x\neq 0}(\|Ax\|_\infty/\|x\|_\infty).$ This norm can also be computed using the explicit formula $\|A\|_\infty=\max_{i \in \{1,\cdots,N\}}\sum_{j=1}^N |A^{ij}|$. 
 
 In this article, we will also use the following "linking" norm $$\|A\|_{2 \to \infty}=\sup_{x \neq 0}\frac{\|Ax\|_\infty}{\|x\|_2},$$ called the \textit{two-to-infinity} norm. This norm satisfies the inequality $\|A\|_{2\to \infty} \leq \|A\|_2$ \citep{Cape}.
 
 We denote the identity matrix of size $N$ by $I^N$ and $\delta^N$ the matrix indicator sequence such that $\delta^N_t=I^N$ if $t=0$ and $\delta^N_t=0$ otherwise. We recall that the convolution $*$ of two matrix sequences of compatible sizes is defined as
$(A*B)_t=\sum_{k=0}^t A_{t-k}B_k.$
Since the convolution is associative, one can recursively define $A^{(*m)}=A*A^{(*(m-1))}$ for $m \in \N ^*$ and $A^{*0}= \delta ^N$.

On a probability space $(\tilde \Omega, \mathcal F_\Pi, \PP)$, let $\Pi=(\Pi^{i}_t(\cdot))_{(i,t)\in \N^* \times \N}$ be an independent and identically distributed (\textit{iid}) family of unit-rate Poisson processes. $\Pi$ plays the role of the underlying driving randomness for the network count process, which will be defined in Section \ref{sec: def_pois}. The filtration associated with $(\Pi^{i}_t(\cdot))_{(i,t)\in \N^* \times \N}$ is denoted by $\mathcal F=(\mathcal F_t)_{t\in \N }$. Given a $\lambda \in \R_+$, $\Pi^i_t(\lambda)$ counts the number of points of $\Pi^i_t$ which are in the interval $[0,\lambda]$. It is clear that $\Pi^i_t(\lambda)$ follows a Poisson distribution of parameter $\lambda$, or $\mathcal P(\lambda)$ for short. If $\lambda=\left(\lambda^i\right)_{i \in \N^*}$ is a non-negative sequence, $\Pi_t(\lambda):=\left(\Pi^i_t(\lambda^i)\right)_{(i,t)\in \N^* \times \N}$ has the multivariate Poisson distribution $\mathcal P(\lambda)= \left(\mathcal P(\lambda^i)\right)_{i \in \N^*}$ where the $\mathcal P(\lambda^i)$ are independent Poisson variables.  

Throughout this paper, $\psi$ denotes a non-decreasing Lipschitz function that takes non-negative values. We set $L=\sup_{x\neq y}(|\psi(y)-\psi(x)|/|y-x|)$. This function will be referred to as the \textit{jump rate}.

The set of (scalar, vector or matrix, depending on the context) real summable sequences is denoted by $\ell_1(\N)$. For $\phi \in \ell_1(\N)$, we denote its sum by $\mathcal S(\phi)=\sum_{t=0}^{+\infty} \phi_t.$
The use of the $\|\cdot\|_1$ norm for sequences will be avoided in this paper to avoid any confusion with matrix norms. We add that if we are considering dynamics over a finite time horizon $\{0,\cdots,T\}$ instead of $\N$, it is sufficient to replace $\phi$ with $\phi \mathds 1_{\cdot \le T}$.

\subsection{The stochastic block model (SBM) graph}
\label{sec: def_sbm}
The stochastic block model (SBM) graph is defined as a graph of $N$ nodes,
 each of which belongs to a single class $C(i)$ and that there are $q\in \N^*$ different classes. For a given class $\ell \in \{1,\cdots,q\}$ we call $N_\ell$ the number of its agents. It is clear that $\sum_{\ell=1}^q N_\ell = N$. We suppose that the fraction of agents in the class $\ell$ remains constant, that is $\alpha_\ell :=(N_\ell/N) \in (0,1)$. We assume that each agent $i$ has an implicit position $x_i=(i/N)\in [0,1]$ which determines its class:
$$C(i)=\min \left\{\ell \in\{1,\cdots,q\} \text{ such that }x_i=\frac{i}{N}\leq \alpha_\ell \right\}.$$

On a probability space, $(\Omega,\mathcal F_U,\Pb)$, let $(U_{i,j})_{(i,j) \in (\N^*)^2}$ be a family of independent and identically distributed \textit{iid} random variables following the $\mathcal U([0,1])$  distribution. We denote the expected value with respect to $\Pb$ by $\Eb$. Given a family $(p_{\ell \to h})_{(\ell, h) \in \{1,\cdots,q\}^2}$ of numbers in $[0,1]$, we build a graph $\mathcal G = (\{1,\cdots,N\},(\xi_{j\to i })_{(i,j)\in \{1,\cdots,N\} ^2})$ with the random edges 
$$\xi^{j\to i}=\mathds 1_{U_{i,j}\leq p^{C(j) \to C(i)}}.$$
This means that a member of the class $\ell$ forms a directed edge with a member of class $h$ with probability $p^{\ell\to h}$, and does not form an edge with probability $1-p_{\ell\to h}$. Note that in general the graph is not necessarily symmetric, not even in distribution. The adjacency matrix of $\mathcal G$ is denoted by $\Xi =\left( \xi^{j\to i}\right)_{(i,j)\in \{1,\cdots,N\}^2}$. Its expected value $$\Eb \left[\Xi\right] =\left( p^{C(j)\to C(i)}\right)_{(i,j)\in \{1,\cdots,N\}^2}$$ is clearly a block matrix and can be factorized as we will see later. 

We now state a crucial result from the random matrix theory which says that as the graph size $N$ grows, the random adjacency matrix $\Xi$ does not deviate too far away from its expected value. 
\begin{Lemma}
    \label{lmm:vershynin}
    Let $N\in \N^*$. There exists an absolute constant $C$ that does not depend on $N$ such that the event $E^N \subset \Omega$ defined by 

    $$E^N= \left \{\left \|\Xi-\Eb \left[\Xi\right] \right \|_2 \leq C \sqrt N \right \}$$
     has a probability $\Pb$ larger than $1-2\mathrm e^{-N}$. 
\end{Lemma}
\begin{proof}
    This Lemma is a direct application of Theorem 4.4.5 in \cite{vershynin}. Indeed, the entries of our matrix are independent and of zero mean. Moreover, since we have in total $q \times q$ centered Bernoulli distributions, the maximum Sub-Gaussian norm $\max_{(i,j) \in \{1,\cdots,N\}^2}\|\left(\Xi-\Eb \left[\Xi\right]\right)^{ij}\|_{\psi^2}$ is bounded by a quantity that does not depend on $N$.
\end{proof}

\subsection{Poisson Autoregression on an SBM graph}
\label{sec: def_pois}
We now fix one realisation $\mathcal G=\left(\{1,\cdots,N\}, \Xi\right)$ of the aforementioned SBM; a graph of $N$ nodes of which each node $i$ belongs to a single class $C(i)\in \{1,\cdots,q\}$. 

Let $Y^N=(Y^{i,N}_t)_{(i,t)\in \{1,\cdots,N\}\times \N}$ be a count time series, i.e. a time series that takes only non-negative integer values. For example, $Y^N_t$ can count the number of spikes in a neuronal network, the number of posts made by users of a social media platform, or the number of orders placed by a financial agent in the time interval $(t-1,t]$. We say that $Y^N$ is a (nonlinear) Poisson autoregression on $\mathcal G$ if, for any $t\ge 1$, $Y_t$ follows a multivariate Poisson distribution, that is, 
\begin{equation*}
    Y^{i,N}_t \big | Y^N_0,\cdots,Y^N_{t-1} \sim \mathcal P (\lambda^{i,N}_t), \quad \text{for }i \in \{1,\cdots,N\}
\end{equation*}
where $\lambda^N=(\lambda^{i,N}_t)_{(i,t)\in \{1,\cdots,N\}\times \N}$ is the intensity time series defined by 
\begin{equation*}
    \lambda ^{i,N}_t=\psi \left(\mu^{C(i)}+\frac{1}{N} \sum_{j=1}^N \xi^{j\to i} \sum_{k=0}^{t-1} \phi ^{C(j)\to C(i)}_{t-k}Y^{j,N}_k\right), \quad \text{for }i \in \{1,\cdots,N\}
\end{equation*}
with the initial values $\lambda^{i,N}_0=\psi \left( \mu_{C(i)}\right)$ and $Y_0^{i,N} \sim \mathcal P (\lambda^{i,N}_0)$, for $i \in \{1,\cdots,N\}$.

Here, $\phi^{\ell \to h}$ is an $\ell_1(\N)$ kernel that encodes the impact that a node $j$ of class $\ell$ has on node $i$ of class $h$, provided that the edge $\xi_{j\to i}$ exists. For instance, if $\phi^{\ell \to h}$ is positive then the observation of non-zero counts $(Y^{j,N}_k)_{k=1,\cdots,t}$ in the past increases $\lambda^{i,N}_{t+1}$ which in turn means that a non-zero count on the node $i$ at time $t+1$ is more likely to occur. In this circumstance, we say that there is cross excitation from class $\ell$ to class $h$. Conversely, if $\phi^{\ell \to h}$ is negative, the realization of non-zero counts in nodes $j$ of class $\ell$ decreases $\lambda^{i,N}_{t+1}$ and therefore the likelihood of observing a non-zero counts in node $i$ at time $t+1$ decreases. In this case, we say that there is cross inhibition from class $\ell$ to class $h$.

The vector $\mu =\left( \mu^1,\cdots, \mu ^q\right)$ is referred to as the baseline pre-intensity. The higher a value $\mu ^h$ is, the more likely it is to observe spontaneous counts of any node that belong to class $h$, that is, counts that are not provoked by the past of that node or the past of nodes which are connected to it. 
\begin{Remark}
    If the kernel $\phi$ as well as the baseline pre-intensities $\mu$ are non-negative, then the jump rate can be chosen $\psi(x)=x$. This particular choice yields the linear Poisson autoregression, which does not allow for inhibition. However, it allows us to have explicit recursive formulae for $\E \left[Y_t\right]$ as well as to perform the likelihood approximation studied in Section \ref{sec:application}.
\end{Remark}

In this article, we construct $Y^N$ by thinning from the Poisson process family $\Pi$ defined in Section \ref{sec:def_general}. The family $\Pi$ should be seen as the source of Poissonian randomness that drives $Y^N$, and will be used as the same driving randomness for the large network limit. The Poisson autoregression $Y^N$ and its intensity $\lambda ^N$ can thus be constructed in matrix form following the recursion
\begin{equation}
    \label{eq: def_pois_auto}
    \begin{cases}
        Y^N_t &=\Pi_t\left(\lambda ^N_t\right)\\
        \lambda ^N_t &=\psi \left(\mu ^N+\sum_{k=0}^{t-1}\Phi ^N_{t-k}Y^N_k\right)\\
        &=\psi \left(\mu ^N+\sum_{k=0}^{t-1}\Phi ^N_{t-k}\Pi_k \left(\lambda^N_k\right)\right),
    \end{cases}
\end{equation}
for $t \in \N$. The effect of the network is now compactly encoded in the random matrix kernel 
\begin{equation}
    \label{def: Phi}
    \left(\Phi^{ij,N}_t\right)_{(i,j)=\{1,\cdots,N\}} := \left(\frac{1}{N}\xi^{j \to i }\phi^{C(j)\to C(i)}_t\right)_{(i,j)\in\{1,\cdots,N\}^2}.
\end{equation}
The pre-intensity vector is defined as $\mu ^N =(\mu^{C(i)})_{i \in \{1,\cdots, N\}}$.

Before proceeding to the study of the large $N$ limit, we comment on the stability of this autoregression. In fact, due to the presence of self and cross excitations, we can see an explosive behaviour in $Y^N$ and $\lambda ^N$. It was shown in \cite{khabou2025markov} that if 
\begin{equation}
\label{ineq: assumption}
    L \rho \left( \mathcal S \left (\left |\Phi^N \right|\right)\right)<1
\end{equation}
holds true\footnote{In the present setting there is randomness that comes from the graph $\mathcal G$. This becomes deterministic once it is sampled.}, then $(Y^N, \lambda ^N)$ converges towards a unique stationary and ergodic time series $(\tilde Y^N, \tilde  \lambda ^N)$ which solves the equation 
\begin{equation*}
    \begin{cases}
        \tilde Y^N_t &= \tilde \Pi _t(\tilde \lambda ^N_t)\\
        \tilde \lambda ^N_t &=\psi \left(\mu ^N + \sum_{k=-\infty}^{t-1} \Phi ^N_{t-k} \tilde Y^N_k \right),
    \end{cases}
\end{equation*}
for $t \in \Z$, where $\tilde \Pi=\left(\tilde \Pi _t\right)_{t \in \Z}$ is a family of \textit{iid} unit-rate Poisson processes such that $\tilde \Pi _t = \Pi _t$ for $t \in \N$. It is possible to show in this case that $$\E \left[Y^N_t\right]=\E \left[\lambda^N_t\right] \preceq \left(I^N-L\mathcal S\left(\left | \Phi^N\right |\right)\right)^{-1} \psi \left(\mu ^N\right),\quad \text{for }t \in \N,$$
where the matrix $ I^N-L \mathcal S \left (\left |\Phi^N \right|\right)$ is invertible if \eqref{ineq: assumption} holds.

We now define a discrete analog of the Poisson process with a deterministic intensity. We will then prove this process to be the limit of the process \eqref{eq: def_pois_auto} as $N$ goes grows large. 
\subsection{Heuristics and discrete Poisson processes}
\label{sec:heuristics}
As we hinted in the previous sections, we suspect the Poisson autoregression approaches a Poisson process with a deterministic intensity. The heuristics are the following: first, for a given node $i$ in the class $h$, we write its intensity while grouping the nodes that belong to the same class together
\begin{align*}
\lambda^{i,N}_t=& \psi \left(\mu ^{C(i)}+\frac{1}{N} \sum_{j=1}^N \xi^{j\to i} \sum_{k=0}^{t-1} \phi ^{C(j)\to C(i)}_{t-k}Y^{j,N}_k\right)\\
=&\psi \left(\mu ^h+\sum_{\ell=1}^q\frac{\alpha _\ell}{N_\ell}\sum_{j:C(j)=\ell}\xi^{j \to i } \sum_{k=0}^{t-1} \phi^{\ell \to h }_{t-k}Y^{j,N}_k\right)
\end{align*}
for all $i$ in class $h$, where $\alpha_\ell = N_\ell/N$. Rearranging the sums gives
$$\lambda^{i,N}_t=\psi \left(\mu ^h+\sum_{\ell=1}^q \alpha _\ell \sum_{k=0}^{t-1} \phi^{\ell \to h }_{t-k} \frac{1}{N_\ell}\sum_{j:C(j)=\ell}\xi^{j \to i }Y^{j,N}_k\right).$$
For all the nodes $j$ that belong to the same class $\ell$, we expect all the counts $Y^{j,N}_k = \Pi _k (\lambda^{j,N}_k)$ have a similar (conditional) Poisson distribution which entails $$\frac{1}{N_\ell}\sum_{j:C(j)=\ell} Y^{j,N}_k \simeq \frac{1}{N_\ell}\sum_{j:C(j)=\ell} \lambda^{j,N}_k,$$ because the expected value of a Poisson distribution is its intensity. Similarly, the law of large numbers suggests that the intensities will ``average out'' to a deterministic quantity whose characteristics depend only on the parameters of the common class $\ell$: $$ \frac{1}{N_\ell}\sum_{j:C(j)=\ell} \lambda^{j,N}_k  \simeq \bar \lambda_k^\ell.$$
Therefore, taking into account the independence of the two sources of randomness one would expect that 
$$\frac{1}{N_\ell}\sum_{j:C(j)=\ell} \xi^{j \to i}Y^{j,N}_k \simeq p^{\ell \to h } \bar \lambda ^\ell _k.$$
Hence, the stochastic intensity $\lambda ^{i,N}_t$ is likely to converge towards a deterministic intensity that depends only on node $i$'s class, more specifically: 
$$
    \lambda^{i,N}_t
    \simeq \psi \left(\mu ^h+\sum_{\ell=1}^q \alpha _\ell p^{\ell \to h } \sum_{k=0}^{t-1} \phi^{\ell \to h }_{t-k} \bar \lambda ^\ell _k \right)
    \simeq  \bar \lambda ^h _t.
$$
This prompts us to define the following discrete Poisson process.

First, let $\bar \lambda=(\bar\lambda^h _t )_{(h ,t)\in \{1,\cdots,q\} \times \N}$ be the deterministic regression defined by 
\begin{equation}
    \label{def: lambda_bar}
    \bar \lambda ^h _t = \psi \left(\mu ^h+ \sum_{\ell=1}^q \alpha _\ell p^{\ell \to h } \sum_{k=0}^{t-1} \phi^{\ell \to h }_{t-k} \bar \lambda ^\ell _k\right),
\end{equation}
which is then cast into the $N$-dimensional vector-valued sequence $\bar\lambda^N:=(\bar\lambda^{i,N}_t)_{(i,t)\in\{1,...,N\}\times \N}$, where $\bar \lambda ^{i,N}_t:=\bar \lambda ^{C(i)}_t$ for $i\in \{1,\cdots,N\}$ and $t\in \N$. 

We note that it is trivial to show that  $\bar \lambda^{N}$ solves 
\begin{equation}
    \label{eq:bar_broadcast}
    \bar \lambda^{i,N}_t= \psi \left(\mu_{C(i)}+\sum_{\ell=1}^q \alpha_\ell \frac{1}{N_\ell}\sum_{j: C(j)=\ell}  p^{\ell \to C(i)} \sum_{k=1}^{t-1} \phi ^{\ell\to C(i)}_{t-k}\bar{\lambda}^{j,N}_k\right),
\end{equation}
because $\bar{\lambda}^{j,N}_k$ are equal for nodes $j$ in the same class. The discrete Poisson process is then defined by thinning from the same driving randomness as the Poisson autoregression by 
\begin{equation}
    \label{def: Y_bar}
    \bar Y^{i,N}_t:= \Pi^{i}_t\left(\bar \lambda^{i,N}_t\right)\quad \text{for } i\in \{1,\cdots,q\} \quad \text{and }t\in \N.
\end{equation}
If we manage to prove that $\lambda^{i,N}_t \simeq\bar \lambda^{i,N}_t = \bar \lambda^{C(i)}_t$, we would expect that $\bar Y^{i,N}_t\simeq Y^{i,N}_t$. 

Before we prove the convergence of the Poisson autoregression defined by Equation \eqref{eq: def_pois_auto} towards the discrete Poisson process defined by Equations \eqref{def: lambda_bar} and \eqref{def: Y_bar} as the number of nodes in the SBM graph grows, we study the behaviour of the limiting object in time. 

\begin{Assumption}
    \label{ass:stability_MF}
    If $\bar M \in \mathcal M_q(\R_+)$ denotes the matrix 
    $$\bar M_{h\ell}=L\alpha_\ell p_{\ell \to h} \mathcal S \left(\left |\phi^{\ell \to h} \right |\right),$$
    then $\bar M$ has a spectral radius strictly less than one ($\rho(\bar M)<1$).
\end{Assumption}
This assumption is a contraction-type argument that proves that the deterministic intensity defined in \eqref{def: lambda_bar} converges as $t$ goes to infinity to some finite constant. This means that after some burn-out period, $\bar Y^N$ can be seen as a homogeneous discrete time Poisson process. 
\begin{Proposition}
\label{prop:cvg_bar_k}
    If Assumption \ref{ass:stability_MF} holds, then the equation
    $$\lambda_\infty^h=\psi \left(\mu_h + \sum_{\ell=1}^q \alpha_\ell p_{\ell \to h} \mathcal S\left(\phi^{\ell \to h}\right) \lambda_\infty^\ell \right ),\quad h\in \{1,\cdots,q\}$$
    has a unique solution in $ \R_+^q$. Moreover, the sequence $(\bar\lambda_t)_{t\in \N}$ defined recursively in Equation \eqref{def: lambda_bar} converges to $\lambda_\infty$ as $t\to +\infty$. 
\end{Proposition}
\begin{proof}
    The proof of this Proposition can be found in Section \ref{proof: cvg_bar_k}.
\end{proof}
\section{Convergence of the SBM Poisson autoregression as $N$ grows to infinity}\label{sec:convergence-N}
Recall that our model has two sources of randomness: (i) the adjacency graph is itself random following the distribution $\Pb$; (ii) the Poisson autoregression on a realisation of the aforementioned graph is driven by the \textit{iid} Poisson processes of distribution $\PP$. We therefore propose an intermediary network process between the SBM Poisson autoregression defined by Equations \eqref{eq: def_pois_auto} which has a doubly stochastic intensity from one side, and the discrete Poisson process \eqref{def: Y_bar} of purely deterministic intensity \eqref{def: lambda_bar}. 

This intermediary is denoted by $\hat \lambda ^N = (\hat \lambda^{i,N}_t)_{(i,t) \in \{0,\cdots,N\}\times \N}$ on the quenched graph and is recursively defined by
\begin{align}
    \hat \lambda^{i,N}_{t+1}&:=\psi \left(\mu_{C(i)}+\frac{1}{N} \sum_{j=1}^N \xi_{j\to i} \sum_{k=1}^t \phi ^{C(j)\to C(i)}_{t+1-k}\hat{\lambda}^{j,N}_k\right)\label{def:lambda_hat} \\
    &=\psi \left(\mu_{C(i)}+\sum_{\ell=1}^q \alpha_\ell \frac{1}{N_\ell}\sum_{j: C(j)=\ell}  \xi_{j\to i} \sum_{k=1}^t \phi ^{\ell\to C(i)}_{t+1-k}\hat{\lambda}^{j,N}_k\right) \nonumber.
\end{align}
In this process, the only source of randomness is the SBM graph, as can be seen from the presence of $\xi_{j \to i}$ and the absence of $\Pi^i$.

The intuition here is that conditioned on the event $E^N$ defined in Lemma \ref{lmm:vershynin}, the adjacency matrix $\Xi = (\xi ^{j\to i})_{(i,j)\in \{1,\cdots,N\}^2}$ is very close to its expectation $\Eb \left[\Xi\right]$, when the number of nodes $N$ is large. This will yield, with a probability higher than $1-2\mathrm e^{-N}$, that $\hat \lambda^{i,N}_t \simeq \bar \lambda^{i,N}_t$, because of the similarity between Equations \eqref{eq:bar_broadcast} and \eqref{def:lambda_hat}. The other half of the proof, that is that $\hat \lambda^{i,N}_t \simeq  \lambda^{i,N}_t$ for a given realisation of the graph will be proven using mean field approximation techniques. 

Before giving our convergence results, we state the following crucial assumption.
\begin{Assumption}
\label{ass: crucial}
    We have that $\nu <1 $, where 
    $$\nu =L\max_{h\in \{1,\cdots,q\}} \alpha_h \left\| \left(p_{h \to \ell}\mathcal S \left(\left|\phi^{h \to \ell}\right |\right)\right)_{(\ell,h)\in \{1,\cdots,q\}^2}\right\|_2.$$
\end{Assumption}
We note that Assumption \ref{ass: crucial} is stronger than \ref{ass:stability_MF} in two ways:
\begin{enumerate}
    \item It involves an upper bound on the Euclidean norm $\|\cdot\|_2$ which dominates the spectral norm.
    \item It involves an upper bound on the maximum proportion of nodes multiplying the kernels instead of the weighted effect of the kernels. 
\end{enumerate}

As we will see in the next section, we will need the even stronger stability assumption to prove time-independent upper bounds.

\begin{Assumption}
\label{ass:crucial_strong}
    We have that $\nu_{\mathcal S}<1$, where 
    $$\nu_{\mathcal S} =L\max_{h\in \{1,\cdots,q\}} \alpha_h \mathcal S \left(\left\| \left(p_{h \to \ell}\left|\phi^{h \to \ell}\right |\right)_{(\ell,h)\in \{1,\cdots,q\}^2}\right\|_2\right).$$
\end{Assumption}

\subsection{Convergence in the Euclidean norm}
\label{sec:cvg_euclidean}
We start by proving the convergence in the Euclidean norm because of its presence in Lemma \ref{lmm:vershynin}. 
First, we show that, if the aforementioned stability condition is in force, the auxiliary intensity $\hat \lambda ^{N}$ remains bounded in time, conditioned on the event $E^N$, whose probability $\Pb$ is higher than $1-2\mathrm e^{-N}$.

We recall that in the quenched case (the graph's randomness in frozen), $\hat \lambda ^N$ is a deterministic time series. 

\begin{Lemma}
\label{lemma:hat_lambda}
   Let $N$ be a fixed number of nodes in $\N^*$ and $T  \in \N^*$ be the time horizon. If Assumption \ref{ass: crucial} holds, then, there exist two positive constants $C$ and $c$ that do not depend on $N$ nor $T$ such that on the event $E^N$ (such that $\Pb (E^N) \ge 1-2\mathrm{e}^{-N}$) we have that 
    $$\left \|\hat \lambda ^{N}_{t}\right\|_2 \leq C \sqrt{N} \exp\left( c\frac{T}{\sqrt N}\right)$$
    for all $t=1,\cdots,T$. 
       If Assumption \ref{ass:crucial_strong} holds, then there exists a threshold $N_0$ that does not depend on time such that for any $N\ge N_0$
    $$\left \|\hat \lambda_t^{N} \right \|_2 \le C \sqrt N,$$
    for all $t \in \N$ on $E^N$.
\end{Lemma}
\begin{proof}
    The proof of this Lemma can be found in Section \ref{proof: hat_lambda}.
\end{proof}

We are now ready to give an upper bound on the average (in $\E$) distance between the SBM Poisson autoregression intensity $\lambda ^{N}$ given in \eqref{eq: def_pois_auto} and the auxiliary intensity $\hat \lambda^{N}$ given in \eqref{def:lambda_hat}. 
\begin{Proposition}
\label{prop:lambda_hat}
    Let $N$ be a fixed number of nodes in $\N^*$ and $T  \in \N^*$ be the time horizon. If Assumption \ref{ass: crucial} holds, then, there exist two positive constants $C$ and $c$ that do not depend on $N$ nor $T$ such that on the event $E^N$ (such that $\Pb (E^N) \ge 1-2\mathrm{e}^{-N}$) we have that 
    $$\left\|\E \left[\left| \lambda^{N}_{t}-\hat \lambda^{N}_{t}\right|\right] \right\|_2\leq C \exp\left(c \frac{T}{\sqrt N}\right)$$
    for all $t=1,\cdots,T$. 

    If Assumption \ref{ass:crucial_strong} holds, then there exists a threshold $N_0$ that does not depend on time such that for any $N\ge N_0$
     $$\left\|\E \left[\left| \lambda^{N}_{t}-\hat \lambda^{N}_{t}\right|\right] \right\|_2\leq C,$$
     for all $t \in \N$ on $E^N$.
\end{Proposition}
\begin{proof}
    The proof of this Proposition can be found in Section \ref{proof:lambda_hat}.
\end{proof}
The next ingredient we need is a control on the distance between the auxiliary intensity process $\hat \lambda ^{N}$ whose single source of randomness is the SBM graph and the fully deterministic intensity $\bar \lambda ^{N}$. 
\begin{Proposition}
\label{prop:hat_bar}
    Let $N$ be a fixed number of nodes in $\N^*$ and $T  \in \N^*$ be the time horizon. If Assumption \ref{ass: crucial} holds, then, there exist two positive constants $C$ and $c$ that do not depend on $N$ nor $T$ such that on the event $E^N$ (such that $\Pb (E^N) \ge 1-2\mathrm{e}^{-N}$) we have that 
     $$\left \| \hat \lambda^{N}_t-\bar \lambda^{N}_t\right\|_2 \leq C T \exp \left ( c\frac{T}{\sqrt N}\right)$$
    for all $t=1,\cdots,T$.
    
    If instead of taking the uniform error in time we bound the time average, we have that
     $$\frac{1}{T}\left \| \sum_{t=1}^T \left | \hat \lambda^{N}_t-\bar \lambda^{N}_t \right | \right\|_2 \leq C \exp \left ( c\frac{T}{\sqrt N}\right).$$
      If Assumption \ref{ass:crucial_strong} holds, then there exists a threshold $N_0$ that does not depend on time such that for any $N\ge N_0$
     $$\left\| \hat \lambda^{N}_{t}-\bar \lambda^{N}_{t} \right\|_2\leq C,$$
     for all $t \in \N$ on $E^N$.
\end{Proposition}
\begin{proof}
    The proof of this Proposition can be found in Section \ref{proof:hat_bar}.
\end{proof}
The last two propositions show that under the weaker stability Assumption \ref{ass: crucial}, the sum of squares of the doubly stochastic intensity $\lambda^{N}$ and the purely deterministic $\bar \lambda ^{N}$ remains bounded by something of order $O(1)$ as soon as $N \sim T^2$. The lack of such a guarantee and the direct use of the discrete Gr\" onwall Lemma would yield an upper bound of order $O(\mathrm e^{cT})$, which does not improve as the network's size grows to infinity. We also emphasize that under the stronger stability Assumption \ref{ass:crucial_strong}, such bounds are valid for all instants in time.

We now use the fact that two count random variables thinned from the same underlying Poisson randomness $\Pi$ are close as soon as their intensity processes are similar. 

\begin{Theorem}
\label{thm:main_l_2}
    Let $N$ be a fixed number of nodes in $\N^*$ and $T  \in \N^*$ be the time horizon. We recall that $\bar Y^{N}$ is defined in \eqref{def: Y_bar}. If Assumption \ref{ass: crucial} holds, then, there exist two positive constants $C$ and $c$ that do not depend on $N$ nor $T$ such that on the event $E^N$ (such that $\Pb (E^N) \ge 1-2\mathrm{e}^{-N}$) we have that 
    $$\frac{1}{N}\sum _{i=1}^N \E \left[ \left| Y_t^{i,N}-\bar Y^{i,N}_t\right|\right]\leq \frac{CT}{\sqrt N} \exp \left ( c\frac{T}{\sqrt N}\right),$$
    for all $t=1,\cdots,T$. 
     If instead of taking the uniform error in time and we bound the time average, we have that
     $$\frac{1}{T}\sum_{t=1}^T\frac{1}{N}\sum _{i=1}^N \E \left[ \left| Y_t^{i,N}-\bar Y^{i,N}_t\right|\right]\leq \frac{C}{\sqrt N} \exp \left ( c\frac{T}{\sqrt N}\right).$$
       If Assumption \ref{ass:crucial_strong} holds, then there exists a threshold $N_0$ that does not depend on time such that for any $N\ge N_0$
       $$\frac{1}{N}\sum _{i=1}^N \E \left[ \left| Y_t^{i,N}-\bar Y^{i,N}_t\right|\right]\leq \frac{C}{\sqrt N} ,$$
       for all $t \in \N$ on $E^N$.
\end{Theorem}
\begin{proof}
The proof of this Theorem can be found in Section \ref{proof:main_l_2}.
\end{proof}
To conclude this section, we point out the height of the likelihood of being on the event $E^N$. For instance, for a small SBM graph of size $N=12$ we have that $\Pb (E^{12})= 1-2 \mathrm e^{-12}\simeq 0.99999.$ If we consider a moderately sized network $N=120$, we have that $\Pb (E^{120})= 1-2 \mathrm e^{-120}\simeq 1-2\cdot 10 ^{-52} \simeq 1.$
This means that effectively for virtually all applications we are almost always under $E^N$.
\subsection{Convergence in the uniform norm}
\label{sec:cvg_uniform}
Theorem \ref{thm:main_l_2} shows that at a given instant in time, the different nodes of the network are \textit{on average} close to a Poisson process with a deterministic intensity. Such a result does not exclude the scenario under which a few nodes differ significantly from their mean field approximation, with the significant difference being "averaged out" by the rest of the nodes.

To have a uniform bound which works for every node, we must obtain results in the $\|\cdot\|_\infty$ norm rather than the $\|\cdot \|_2$ norm. However, we cannot use the $\|\cdot\|_\infty$ norm directly on the centered matrix $\Phi^N_k-\Eb \left[\Phi^n_k\right]$, because the presence of the absolute value in its expression prevents the randomness from averaging out by the law of large numbers. This is why we use the \textit{two-to-infinity} norm, in conjunction with the results we have already proven for the two norm to obtain bounds that work uniformly for all the nodes. 

We recall that the \textit{two-to-infinity} norm of a matrix $A$ is defined as 
$$\left \| A\right \|_{2\to \infty}:=\sup_{x \in \R^N/\{0\}}\frac{\|Ax\|_{\infty}}{\|x\|_2}.$$

We start by proving a result on the boundedness of the application of $\Phi^N-\Eb \left[\Phi^N\right]$ on $\bar \lambda ^{N}$ in the uniform norm using more classical probability theory results. As we hinted above, we avoid dealing with $\Phi^N-\Eb \left[\Phi^N\right]$ separately because its uniform norm is not small. 

\begin{Proposition}
    \label{prop:union_bound}
     Fix $\varepsilon \in (0,1/2)$ and let $N\geq \left[\ln (8/3)/(2 \alpha_{\min})\right]^{1/(2\varepsilon)}$ be the network size. Then on the event $F^N$, whose probability $\Pb$ larger than $1-4Nq \exp (-2\alpha_{\max}N^{2\varepsilon})$, we have that
    $$\left \| \left (\Eb [\Phi^N_{t-k+1}]-\Phi^N_{t-k+1}\right) \bar \lambda^{N}_k \right \|_\infty \leq L \max_{h,\ell=1,\cdots,q} \left | \phi^{\ell\to h}_{t-k+1} \bar \lambda^\ell_k\right| \frac{1}{\sqrt {N^{1-2\varepsilon}}}, $$
for any $t\in \N$ and $k=1,\cdots,t$.
\end{Proposition}
\begin{proof}
    The proof can be found in Section \ref{proof:union_bound}.
\end{proof}
We note that the guarantees are worse than the guarantees obtained for the $\|\cdot\|_2$ norm proven in Section \ref{sec:cvg_euclidean}. Indeed, this can be seen in two aspects: (i) the probability $\Pb$ is $1-4Nq \exp (-2\alpha_{\max}N^{2\varepsilon})$ which is clearly lower than $1-2\exp(-N)$, especially if the tolerance $\varepsilon$ is small; (ii) the speed itself is slower, $O(N^{-1/2+\varepsilon})$ compared to $O(N^{-1/2})$. Of course we can approach this ideal speed by taking a smaller value of $\varepsilon$, but this happens with a smaller probability.

We now state a result which is equivalent to Propositions \ref{prop:lambda_hat} and \ref{prop:hat_bar} adapted to the $\|\cdot\|_\infty$ norm. 
In order to do so, we need the following assumption, which is similar to Assumption \ref{ass: crucial} for the uniform norm.
\begin{Assumption}
    \label{ass:uniform}
    We have that $\nu'<1$ where
    $$\nu' = L\left \|\left(\alpha_h p^{h \to \ell}\mathcal S \left(\left | \phi^{h \to \ell}\right |\right)\right)_{(\ell,h)\in \{1,\cdots,q\}} \right \|_\infty.$$
\end{Assumption}
We also introduce its stronger counterpart, which is analogous to Assumption \ref{ass:crucial_strong}.
\begin{Assumption}
    \label{ass:uniform_strong}
    We have that $\nu'_{\mathcal S}<1$ where
    $$\nu'_{\mathcal S} = L \mathcal S \left(\left \|\left(\alpha_h p^{h \to \ell}\left | \phi^{h \to \ell}\right |\right)_{(\ell,h)\in \{1,\cdots,q\}} \right \|_\infty\right).$$
\end{Assumption}

The event $F^N$ mentioned in the Proposition below exists by virtue of Proposition \ref{prop:union_bound}.
\begin{Proposition}
\label{prop:intensity_uniform}
    Let $T \in \N^*$ be the time horizon. Fix an $\varepsilon \in (0,1/2)$ and let $N\geq \left[\ln (8/3)/(2 \alpha_{\min})\right]^{1/(2\varepsilon)}$ be the network size. If Assumptions  \ref{ass: crucial} and \ref{ass:uniform} hold, then on the event $E^N \cap F^N$, whose probability $\Pb$ is larger than $1-4Nq \exp(-2\alpha_{\max}N^{2\varepsilon})-2\exp(-N)$, we have that 
$$\left \| \E \left [ \left| \lambda^{N}_{t+1}-\bar \lambda^{N}_{t+1}\right |\right]\right \|_\infty \leq C T\left(\frac{N^\varepsilon+T \exp \left(c \frac{T}{\sqrt N}\right)}{\sqrt N}\right)$$
for all $t \in \{0,\cdots,T\}$.

Under the stronger stability assumptions \ref{ass:crucial_strong} and \ref{ass:uniform_strong} we obtain the time independent upper bound 
$$\left \| \E \left [ \left| \lambda^{N}_{t+1}-\bar \lambda^{N}_{t+1}\right |\right]\right \|_\infty \leq \frac{C}{\sqrt {N^{1-2\varepsilon}}},$$
for all $t \in \N$ on $F^N \cap E^N$. 
\end{Proposition}
\begin{proof}
    The proof can be found in Section \ref{proof:intensity_uniform}.
\end{proof}
Despite the fact that under the weaker stability Assumptions \ref{ass: crucial} and \ref{ass:uniform} the mean field approximation is not guaranteed to work unless $N\gg T^4$, this remains clearly better than the guarantee offered by Gr\"onwall's inequality which requires $N\gg \mathrm e^{cT}$.

We conclude this section by giving the upper bound on the uniform approximation of the count series themselves. 

\begin{Theorem}
\label{thm:main_uniform}
    Let $T \in \N^*$ be the time horizon. Fix an $\varepsilon \in (0,1/2)$ and let $N\geq \left[\ln (8/3)/(2 \alpha_{\min})\right]^{1/(2\varepsilon)}$ be the network size. If Assumptions  \ref{ass: crucial} and \ref{ass:uniform} hold, then on the event $E^N \cap F^N$, whose probability $\Pb$ is larger than $1-4Nq \exp(-2\alpha_{\max}N^{2\varepsilon})-2\exp(-N)$, we have that  
$$\max_{i\in \{1,\cdots,N\}} \E \left [ \left| Y^{i,N}_{t+1}-\bar Y^{i,N}_{t+1}\right |\right] \leq C T\left(\frac{N^\varepsilon+T \exp \left(c \frac{T}{\sqrt N}\right)}{\sqrt N}\right)$$
for all $t \in \{0,\cdots,T\}$.

Under the stronger stability Assumptions \ref{ass:crucial_strong} and \ref{ass:uniform_strong} we obtain the time independent upper bound 
$$\max_{i \in \{1,\cdots,N\}} \E \left [ \left| Y^{N}_{t+1}-\bar Y^{N}_{t+1}\right |\right] \leq \frac{C}{\sqrt {N^{1-2\varepsilon}}},$$
for all $t \in \N$ on $F^N \cap E^N$. 
\end{Theorem}
\begin{proof}
    The proof is analogous to that of Theorem \ref{thm:main_l_2} with the use of Proposition \ref{prop:intensity_uniform} instead of Propositions \ref{prop:lambda_hat} and \ref{prop:hat_bar}, and shall therefore be omitted.
\end{proof}

\section{Verification and analysis of convergence results via simulations}
\label{sec:verification}
In this section, we provide an empirical illustration of the theoretical results obtained in Section \ref{sec:convergence-N}. We show how the auxiliary process $\widehat{\lambda}$, the true underlying process $\lambda$ and the mean-field process $\bar{\lambda}$ evolve over time in example community structures. Furthermore, we demonstrate how the error in the mean-field approximation follows the derived theoretical rates. 

\subsection{Simulation setup}

We simulate interaction dynamics on two networks of different sizes (500 and 5000 nodes), each divided into three communities based on predefined proportions $\alpha=[0.5,0.3,0.2]$. We refer to the first community as $C(1)$, the second community as $C(2)$ and the third community as $C(3)$. 

Two types of stochastic block model (SBM) structures are examined: an assortative SBM, where nodes are more likely to connect within their community, and a dissociative SBM, where connections are more likely between different communities. The specific probability matrices that we used are
\[
P_{\mathrm{assortative}}=
\begin{pmatrix}
0.9 & 0.2 & 0.1\\
0.2 & 0.9 & 0.3\\
0.1 & 0.3 & 0.9
\end{pmatrix}
\qquad\text{and}\qquad
P_{\mathrm{dissociative}}=
\begin{pmatrix}
0.1 & 0.9 & 0.7\\
0.9 & 0.1 & 0.9\\
0.7 & 0.9 & 0.1
\end{pmatrix}.
\]
Figure~\ref{fig:adj_matrices} shows an example realisation of the assortative and dissociative SBMs with these probability matrices, with darker color indicating more connections within that specific region. The interaction dynamics are governed by different temporal kernels assigned to each community as outlined below,
\[
\phi^{C(1)\to C(1)}_t=\phi^{C(1)\to C(2)}_t=\phi^{C(1)\to C(3)}_t=0.9e^{-0.5t},
\]
\[
\phi^{C(2)\to C(1)}_t=\phi^{C(2)\to C(2)}_t=\phi^{C(2)\to C(3)}_t=0.5te^{-t^2},
\]
\[
\phi^{C(3)\to C(1)}_t=\phi^{C(3)\to C(2)}_t=\phi^{C(3)\to C(3)}_t=\frac{0.9}{1+t^{2.5}}.
\]

\begin{figure}
    \centering
    \includegraphics[width=0.8\linewidth]{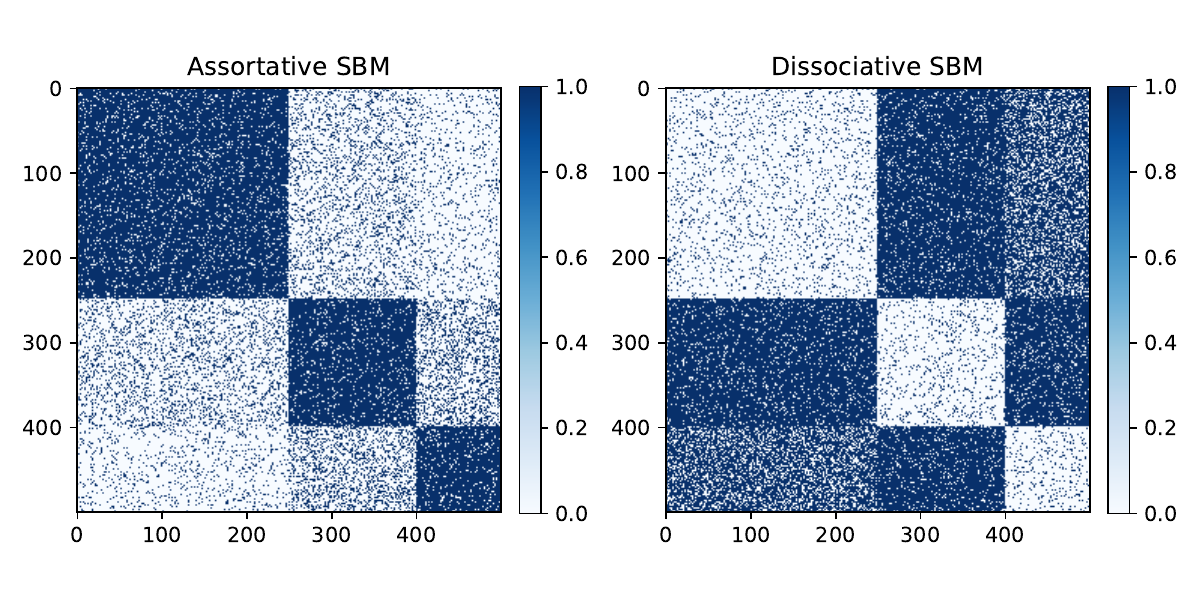}
    \caption{Visualization of the adjacency matrices generated according to the SBM framework, for the assortative example (to the left) and dissociative example (to the right}
    \label{fig:adj_matrices}
\end{figure}

Community $C(1)$ follows an exponential kernel, characterized by a rapid and memoryless decay in which the impact of past events diminishes uniformly; $C(2)$ uses a Rayleigh kernel, which models interactions that are most likely to occur after a short delay, peaking at an intermediate time before decaying; $C(3)$ follows a power law kernel, which decays slowly and allows past events to exert influence far into the future, reflecting long-term memory. These kernels are shown in Figure~\ref{fig:temporal_kernels}. This heterogeneous assignment of temporal kernels allows each community to exhibit distinct dynamical behaviour, and the model design facilitates the study of how temporal dynamics interact with structural patterns in both assortative and dissociative network regimes.

\begin{figure}
    \centering
    \includegraphics[width=\linewidth]{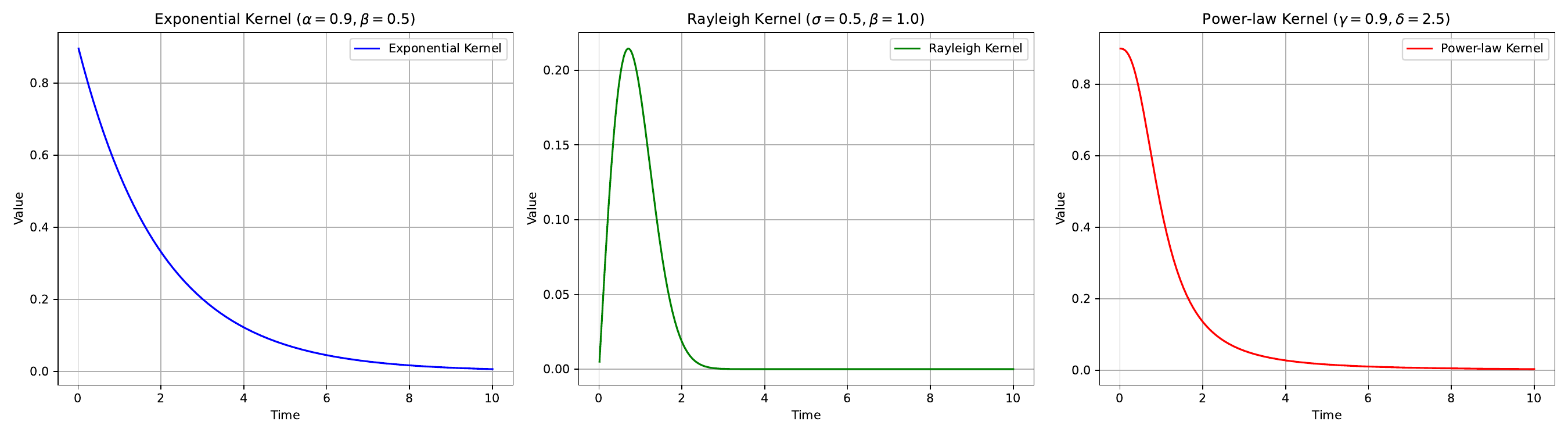}
    \caption{Visualization of the three temporal kernels used in the model. The first panel shows the exponential kernel with parameters \( [0.9,\ 0.5] \), defined as \( \alpha e^{-\beta t} \). The second panel shows the Rayleigh kernel with parameters \( [0.5,\ 1.0] \), given by \( \alpha t e^{-\beta t^2} \). The third panel displays the power-law kernel with parameters \( [0.9,\ 2.5] \), given by \( \alpha / (1 + t^\beta) \).}
    \label{fig:temporal_kernels}
\end{figure}

Figure~\ref{fig:lambdas_sbm} illustrates the evolution of the true intensity $\lambda$, auxiliary process intensity $\widehat{\lambda}$, and mean-field intensity $\bar{\lambda}$ over time across communities and network structures. In all cases, the true intensity is scattered around the auxiliary process intensity (solid lines), and for the larger network the distance between the true intensity and the auxiliary process is much smaller than for the smaller network. For the large network both the true intensity and the auxiliary process intensity align more closely with the mean-field approximation (dotted lines), reflecting the increased average effect at scale. In the assortative SBM (top row), the mean-field approximation is most accurate for the largest community $C(1)$, where the intensity is also the highest; this is consistent with the exponential kernel's sustained self-excitation and high intra-community connectivity. The intensities for $C(2)$ and $C(3)$ are similar in magnitude, with the mean-field approximation slightly more accurate for $C(3)$, likely due to the power law kernel's lower but smoother decay, which introduces less variance than the Rayleigh kernel in $C(2)$.

In the dissociative SBM (bottom row), the true and auxiliary intensities are notably closer to the mean-field across all communities. This improvement is likely due to the broader, more evenly distributed influence from cross-community connections, which enhances the averaging effect, particularly in $C(2)$ and $C(3)$. Consequently, intensity levels rise for $C(2)$ and $C(3)$ compared to the assortative case, as they now receive input from a larger portion of the network. Additionally, the dissociative structure appears to homogenize the temporal dynamics across communities, reducing the distinction between their kernel-induced behaviours.

\begin{figure}[htbp]
    \centering

    \begin{subfigure}[b]{1\textwidth}
        \centering
        \includegraphics[width=\textwidth]{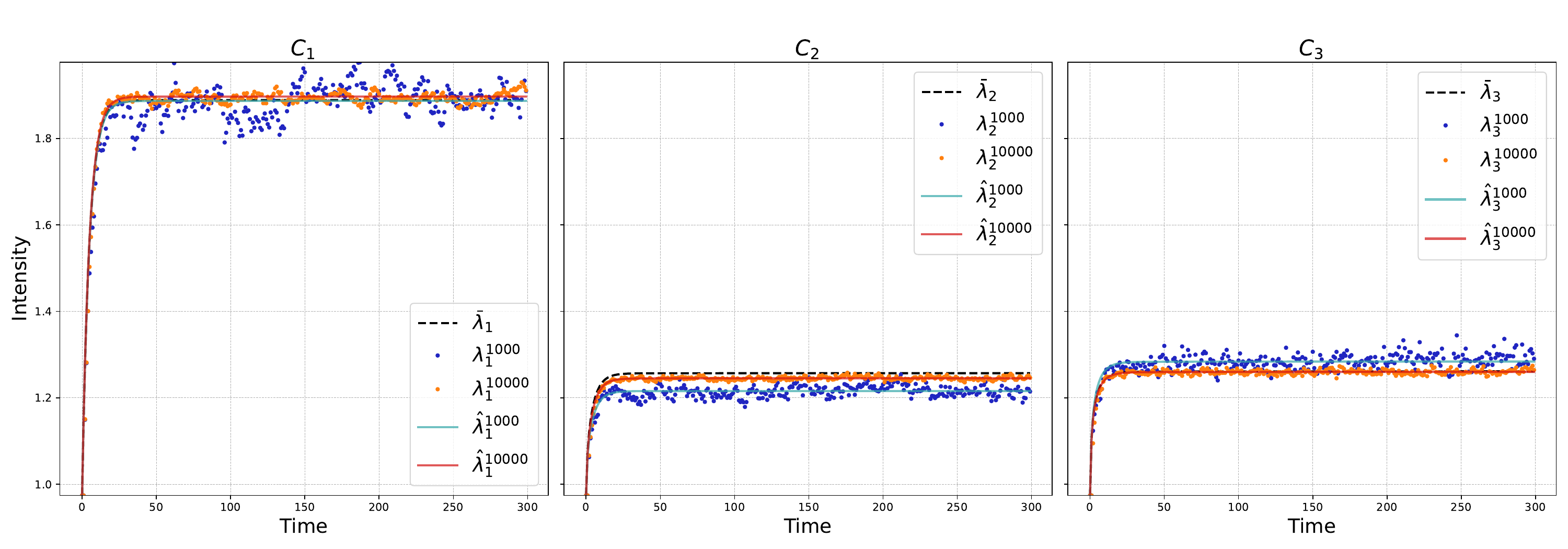}
        \caption{Assortative SBM}
        \label{fig:lambda_assortative}
    \end{subfigure}

     \vspace{1em}

    \begin{subfigure}[b]{1\textwidth}
        \centering
        \includegraphics[width=\textwidth]{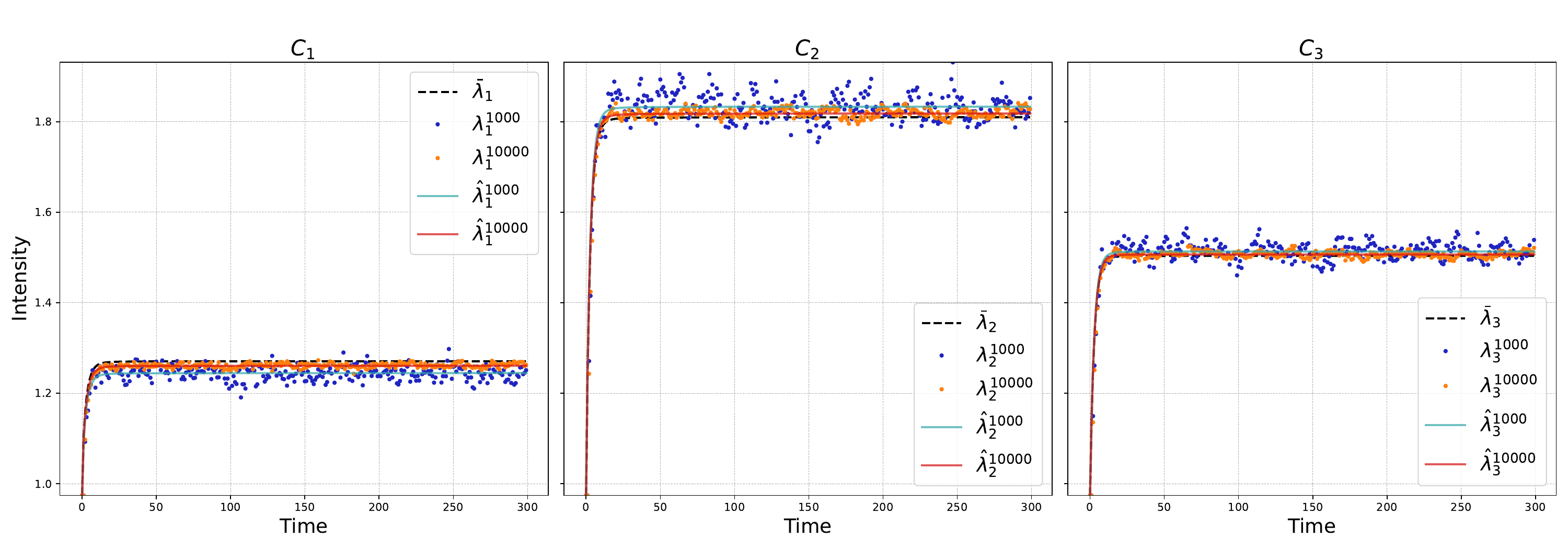}
        \caption{Dissociative SBM}
        \label{fig:lambda_dissociative}
    \end{subfigure}

    \caption{Evolution of the auxiliary process intensity $\hat{\lambda}$, the true underlying process intensity $\lambda$, and the mean-field intensity $\bar{\lambda}$ on the log scale over time, shown separately for the dissociative and assortative SBM structures. The plots illustrate how structural differences in connectivity influence the convergence and spread of the dynamic processes across the network.}
    \label{fig:lambdas_sbm}
\end{figure}

\subsection{Dependence on the number of nodes $N$}

Next, we examine the behavior of the errors with increasing number of nodes $N$ across multiple graph realizations. Figure~\ref{fig:mean-distances-log-n} reports the mean distances between the true counts and counts from the mean-field process, as well as the true intensity and the mean-field intensity, verifying the results of Theorem \ref{thm:main_l_2}. Each line corresponds to a different graph realization, and the results are averaged over 100 different runs. For both comparisons, the errors decrease systematically as the network size increases. The behavior is consistent across the different graph realizations, with only limited variability between instances. On the log--log scale, the observed decay is broadly consistent with the reference slope $-0.5$, in agreement with the canonical $N^{-1/2}$ convergence rate obtained in Theorem \ref{thm:main_l_2}. This shows that the approximation becomes accurate when averaging over nodes as the size of the network increases.

\begin{figure}[htbp]
    \centering
    \includegraphics[width=\textwidth]{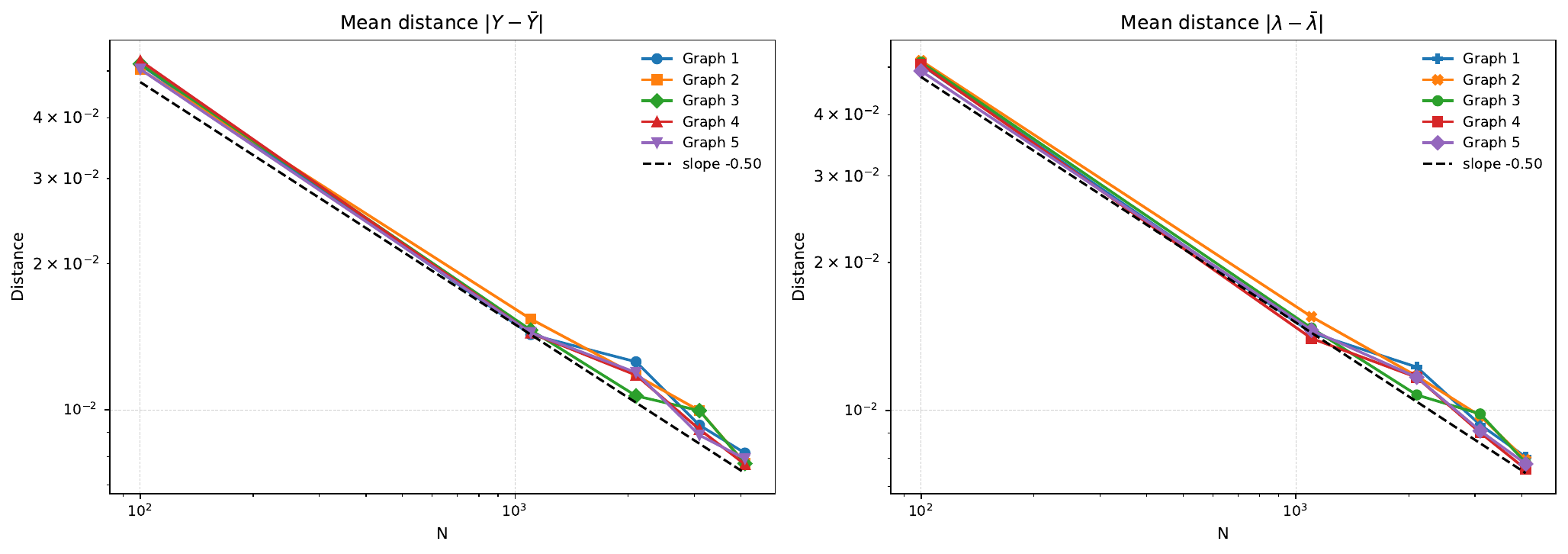}
    \caption{Mean distances as a function of the number of nodes $N$, averaged over 100 runs, for five graph realisations. The left panel shows the mean distance $|Y-\bar{Y}|$ between the true and mean-field processes, and the right panel shows the mean distance $|\lambda-\bar{\lambda}|$. The dashed line gives the reference slope $-0.5$.}
    \label{fig:mean-distances-log-n}
\end{figure}

Figure~\ref{fig:max-distances-log-n} reports the corresponding maximum distances and therefore provides a numerical illustration of the uniform convergence results in Theorem \ref{thm:main_uniform}. As expected, the maximum distances display more variability across graph realizations than the mean distances. Nevertheless, the discrepancies show an overall decreasing trend as $N$ grows. Their behavior is again broadly consistent with the reference slope $-0.5$. These results support the fact that the mean-field approximation does not only perform well on average, but also captures the behaviour of individual nodes increasingly well in larger networks.

\begin{figure}[htbp]
    \centering
    \includegraphics[width=\textwidth]{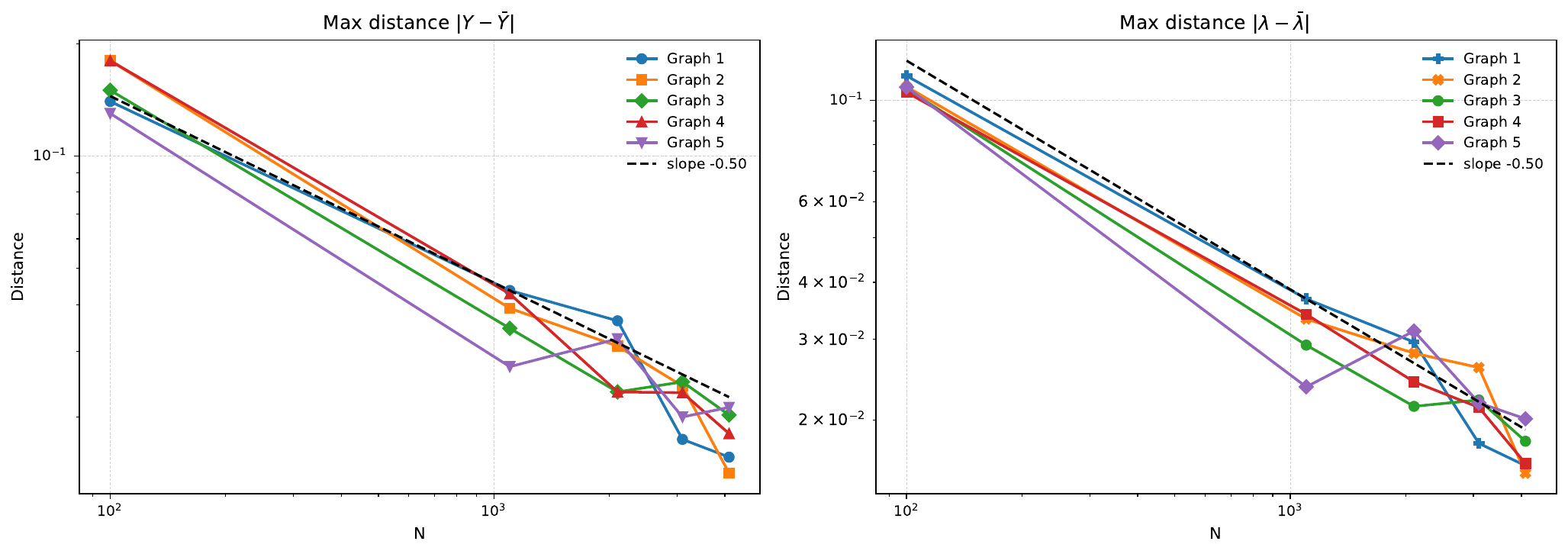}
    \caption{Maximum distances as a function of the number of nodes $N$, averaged over 100 runs, for five graph realisations. The left panel shows the maximum distance $|Y-\bar{Y}|$ between the true and mean-field processes, and the right panel shows the maximum distance $|\lambda-\bar{\lambda}|$. The dashed line gives the reference slope $-0.5$.}
    \label{fig:max-distances-log-n}
\end{figure}

Simulations demonstrating the intermediary convergence results for the true and auxiliary intensities, and the auxiliary and mean-field intensities, are presented in Appendix \ref{sec:add_sims}.

\subsection{Dependence on the time horizon $T$}

Finally, we examine the behaviour of the errors as the time horizon $T$ increases, for a fixed network size. The  results are reported in Appendix \ref{sec:add_sims}. In keeping with the results of Theorem \ref{thm:main_l_2} and Theorem \ref{thm:main_uniform}, over the range of time horizons considered, Euclidean and uniform errors remain of comparable magnitude across the different graph realisations and do not show systematic deterioration. 

%
%
%

\section{Application: Mean-field approach to the parametric inference for linear Poisson autoregressions}
\label{sec:application}
We consider the task of estimating the parameters of a Poisson autoregression on a large network with a stochastic block model structure, leveraging the presented asymptotic mean-field results for significant gains in computational efficiency over the exact likelihood approach. For simplicity, we provide details for the linear Poisson autoregression ($\psi(x)=x$). We consider the general case where the unknown parameters of the model are $\btheta = [\bmu^T,\btheta_\phi^T]^T$, where $\bmu = [\mu_1,...,\mu_q]^T$ and $\btheta_\phi$ is the vector of parameters for the kernels $\{\phi_{t}^{l\rightarrow h}\}_{l,h}$.

The primary assumption for our approach is that the process is in the mean-field state and can therefore be approximated as an inhomogeneous Poisson process with a node belonging to network community $h\in\{1,...,q\}$ having deterministic intensity
$$
\balambda_{t+1}^h =  \mu_h + \sum_{l=1}^q\alpha_l p_{l\rightarrow h}\sum_{k=1}^t\phi_{t+1-k}^{l\rightarrow h} \balambda_k^l \quad \text{for } t\geq 1,
$$
where $\balambda_1^h = \mu^h$. In this setting, the log-likelihood is
	\begin{align*}
	\ell(\btheta)  = \sum_{i=1}^N\sum_{t=1}^T Y_t^i\log(\bar\lambda_t^{C(i)}) - \lambda_t^{C(i)} & = \sum_{h = 1}^q\sum_{i=1}^N \sum_{t=1}^T \mathbb{I}[C(i)=h] \left(Y_t^i\log(\bar\lambda_t^{C(i)}) - \bar\lambda_t^{C(i)}\right)\\
	& = \sum_{h = 1}^q\sum_{t=1}^T \mathcal{Y}_t^h\log(\bar\lambda_t^{h}) - N_h\bar\lambda_t^{h},
	\end{align*}
	where $\mathcal{Y}_t^h = \sum_{i=1}^N\mathbb{I}[C(i) = h]Y_t^i$ is the total counts at time $t$ for community $h$. This can be optimized via gradient descent where $\nabla_\btheta\ell(\btheta) = \sum_{h = 1}^q\sum_{t=1}^T [(\mathcal{Y}_t^h/\bar\lambda_t^{h}) - N_h]\nabla_\btheta\bar\lambda_t^{h}$.

\subsection{Direct estimation of the kernel}
\subsubsection{Formulation}
The first scenario we consider is the direct estimation of each element of the kernels $\{\phi^{l\rightarrow h}\}_{l,h=1}^q$. Here, we naturally have to assume all kernels have finite support (a finite maximal lag), such that $\btheta_\phi$ is a vector of length $\sum_{l,h=1}^q K_{lh}$, where $K_{lh}$ is the finite number of non-zero elements in $\phi^{l\rightarrow h}$ to be estimated. To optimize the likelihood, note
$$
\balambda_{t+1}^h =  \mu^h + \sum_{l=1}^q\alpha^l p^{l\rightarrow h}\sum_{k=1}^{\min\{t,K_{lh}\}}\phi_{t+1-k}^{l\rightarrow h} \balambda_k^l \quad \text{for } t\geq 1,
$$
and
$$
\nabla_\btheta\balambda_{t+1}^h =  \nabla_\btheta \mu^h + \sum_{l=1}^q\sum_{k=1}^{\min\{t,K_{lh}\}} \Big\{\balambda_k^l\nabla_\btheta\left(\alpha^l p^{l\rightarrow h}\phi_{t+1-k}^{l\rightarrow h}\right)  + \alpha^l p^{l\rightarrow h}\phi_{t+1-k}^{l\rightarrow h}\nabla_\btheta \balambda_k^l\Big\} \quad \text{for } t\geq 1.
$$
Population parameters $\alpha^\ell$ and $p^{\ell \to h}$ can be replaced with the plug-in estimators $\hat{\alpha}^{\ell} = N^\ell / N$ and $\hat{p}^{\ell \to h} = \sum_{i : C(i) = \ell}\sum_{j : C(j) = h}\xi^{j \to i} / N^\ell N^h.$ From which, we can then compute the relevant gradients for $t =1, \dots, T-1$ as
\begin{align*}
    \pdv{\balambda^h_{t+1}}{\mu^r} &= \mathbb{I}\{r=h\} + \sum_{\ell=1}^q\hat{\alpha}^\ell \hat{p}^{\ell \to h}\sum_{k=1}^{\min\{t,K_{\ell h}\}}\phi_{t+1-k}^{\ell \to h}\pdv{\balambda^\ell_{t+1-k}}{\mu^r},\\
    \pdv{\balambda^h_{t+1}}{\phi_m^{s \to r}} &= \hat{\alpha}^s \hat{p}^{s \to h} \balambda_{t+1-m}^{s}\mathbb{I}\{m \leq t, r=h\} + \sum_{\ell=1}^q \hat{\alpha}^\ell \hat{p}^{\ell \to h} \sum_{k=1}^{\min\{t,K_{\ell h}\}}\phi_{k}^{\ell \to h}\pdv{\balambda^h_{t+1-k}}{\phi_m^{s \to r}}.
\end{align*}
with $\partial \balambda_1^h / \partial \mu^r = \mathbb{I}\{r = h\}$ and $\partial\balambda_1^h /\partial \phi_m^{s \to r} = 0$.

\subsubsection{Simulations}
Consider a $q=2$ group structure for $N=500$ nodes, with $\alpha^1=\alpha^2=1/2$ giving an equal split between the two groups. The probability matrix is set to be
\begin{equation*}
 P\equiv   \begin{pmatrix}
        p^{1\to 1} & p^{1\to 2}\\
        p^{2\to 1} & p^{2 \to 2}
    \end{pmatrix} = 
    \begin{pmatrix}
        0.9 & 0.8 \\
        0.85 & 0.95
    \end{pmatrix}.
\end{equation*}
For the Poisson autoregression, the vector of baseline intensities is set to $\bmu^T = (0.8, 0.25)$, and we consider a lag of $K_{lh} = 1$ for all $(l,h)\in\{1,2\}^2$, setting
\begin{equation*}
    \begin{pmatrix}
        \phi_1^{1 \to 1} & \phi_1^{1 \to 2} \\
        \phi_1^{2 \to 1} & \phi_1^{2 \to 2}
    \end{pmatrix} 
    = 
    \begin{pmatrix}
        2.5e^{-1.4} & 1.5e^{-0.7}\\
        0.6 e^{-1.15} & e^{-2}
    \end{pmatrix}
    \approx
        \begin{pmatrix}
        0.617 & 0.745\\
        0.190 & 0.135
    \end{pmatrix}.
\end{equation*}
 Over 1000 repeat simulations, the baseline $\bmu^\top$ was estimated with a mean of $(0.800, 0.250)$ and standard deviations $(0.026, 0.015)$. For $\{\phi_1^{r \to s}\}$, we obtain a mean of
\begin{equation*}
\begin{pmatrix}
    0.409 & 0.206 \\
    1.404 & 1.661
\end{pmatrix},
\end{equation*}
with a standard deviation of
\begin{equation*}
\begin{pmatrix}
    0.186 & 0.223 \\
    0.223 & 0.240
\end{pmatrix}.
\end{equation*}
While recovery of the kernel parameters appears poor, the mean-field intensity functions $\balambda^{1}_t$ and $\balambda^2_t$ are recovered with high accuracy and precision. The intensity functions averaged over $t$ are $2.453$ and $2.065$, for groups 1 and 2, respectively. When averaged over time in the same way, the simulations give a mean of $2.460$ and $2.058$, with 95\% simulation envelopes $(2.415, 2.564)$ and $(1.957,2.125)$. This points to underlying structural identifiability issues with the model, where many possible combinations of parameters give rise to similar intensity profiles.

\subsection{Estimation for a basis representation of the kernel}

\subsubsection{Formulation}

An alternative approach, which does not need the assumption of a maximal lag, considers $p$ basis functions for the kernel, shared across time and groups. We assume our kernels to be $\phi_t^{\ell \to h} = \sum_{m=1}^p \gamma^{\ell \to h}_m\exp\{-mt/\tau\}$, with $\tau$ a known, positive integer. The intensity function of node $i$ is then
\begin{equation}
    \lambda_{t+1}^i = \mu^{C(i)} + \frac{1}{N}\sum_{j=1}^N\xi^{j \to i}\sum_{k=1}^t\sum_{m=1}^{p}\gamma_m^{C(j) \to C(i)}\exp\{-m(t+1-k)/\tau\}Y_k^j, 
    \label{eqn:intensity_exp_basis}
\end{equation}
again with $Y_0^i \equiv 0$, and the mean-field intensity becomes
\begin{equation}
    \balambda_{t+1}^h = \mu^h + \sum_{\ell=1}^q\alpha^\ell p^{\ell \to h}\sum_{k=1}^t \sum_{m=1}^{p}\gamma_m^{\ell \to h}\exp\{-m(t+1-k)/\tau\} \balambda_k^\ell,
    \label{eqn:intensity_mean_field}
\end{equation}
with $\balambda_1^h = \mu^h$. The parameter vector $\btheta_\phi$ now contains the weights $\left\{\gamma_m^{l\rightarrow h}\right\}_{(m,l,h)\in\{1,...,p\}\times \{1,...,q\}^2}$.

To simplify inference, we perform gradient ascent only with respect to $\gamma_m^{\ell \to h}$. The parameter $\mu^h$ is estimated separately using the identity $\lambda_1^h = \mu^h$, which yields the unbiased estimator
\begin{equation}
\hat{\mu}^h = \frac{1}{N_h}\sum_{i : C(i) = h} Y_1^i.
\label{eqn:mu_hat}
\end{equation}
Gradients with respect to $\gamma_m^{\ell \to h}$ can then be computed as
\begin{align}
    \pdv{\balambda_{t+1}^h}{\gamma_n^{s \to r}} &= \mathbb{I}\{h=r\}\alpha^s p^{s \to r} \sum_{k=1}^t\exp\{-n(t+1-k)/\tau\}\balambda_k^s\notag \\
    &+ \sum_{\ell = 1}^q \alpha^\ell p^{\ell \to h} \sum_{k=1}^t\sum_{m=1}^p\gamma_m^{\ell \to h} \exp\{-m(t+1-k)/\tau\}\pdv{\balambda_k^\ell}{\gamma_n^{s \to r}},
    \label{eqn:gamma_grad_exp}
\end{align}
and a gradient ascent procedure is implemented with the plug-in estimators $\hat{\mu}^\ell$, $\hat{\alpha}^\ell$ and $\hat{p}^{\ell \to h}$ (for $(\ell,h)\in\{1,2\}^2$).

\subsubsection{Simulations}
Consider a $q=2$ group structure on network with $N=2500$ nodes, with $\alpha^1=\alpha^2=1/2$ giving an equal split between the two groups. We consider a dense network, setting 
\[
p = \begin{pmatrix}
     0.9 & 0.8 \\
     0.85 & 0.95
\end{pmatrix},
\]
with $T=100$. We use $p=3$ basis functions for the kernels with coefficients $\gamma^{\ell \to h} = (\gamma_1^{\ell \to h}, \gamma_2^{\ell \to h}, \gamma_3^{\ell \to h})^T$ given as $\gamma^{0 \to 0} = (1.5, 0.4, 0.2)^T$, $\gamma^{0 \to 1} = (1.1, 0.2, 0.1)^T$, $\gamma^{1 \to 0} = (1.2, 0.25, 0.1)^T$, and $\gamma^{1 \to 1} = (2.0, 0.3, 0.15)^T$. We set the baseline intensity parameters to be $\mu^1 = 0.8$ and $\mu^2=0.25$.

A multi-start gradient-ascent scheme is initialized using a data-driven anchor for the first basis coefficient and randomized values 
for the remaining coefficients. For each community pair $(\ell,h)\in \{1,2\}^2$, we first estimate a one-step kernel coefficient $\phi^{\ell\to h}_{1}$ by nonnegative least squares using the counts from the $t=1$ to $t=2$ transition. This is then mapped to the first exponential-basis coefficient via $\gamma^{\ell\to h}_{1}=\phi^{s\to r}_{1}/\exp(-1)$, clipped to $[\gamma_{\min},\gamma_{\max}]$ to avoid degenerate starts. 

For $m\ge 2$, $\gamma^{\ell \to h}_{m}$ is sampled from $N(0,\;(\kappa\,\gamma^{\ell\to h}_{1})^{2})$, clipped at each draw to $|\gamma^{\ell \to h}_{m}|\le \eta\,\gamma^{\ell \to h}_{1}$, and, if necessary, rescaled the coefficient vector such that $\sum_{m=1}^{p}\gamma^{s\to r}_{m}>\varepsilon$ at initialization, giving a valid kernel at $t=0$. This procedure gives diverse but feasible initial starting points, while retaining a data-informed scale through $\gamma_{1}$. A parameter estimate is computed as the median value from 5000 different start points. 

Figures (\ref{fig:inferred_gamma_median}) and (\ref{fig:inferred_kernel_median}) show that while there is substantial variance in the estimator for the basis coefficients, particularly basis functions 2 and 3, the kernel is still estimated to a high-level of precision. Again, this is in part due to a structural identifiability issue, with many possible basis combinations giving approximately the same kernel.

\begin{figure}
    \centering
    \includegraphics[width=0.75\linewidth]{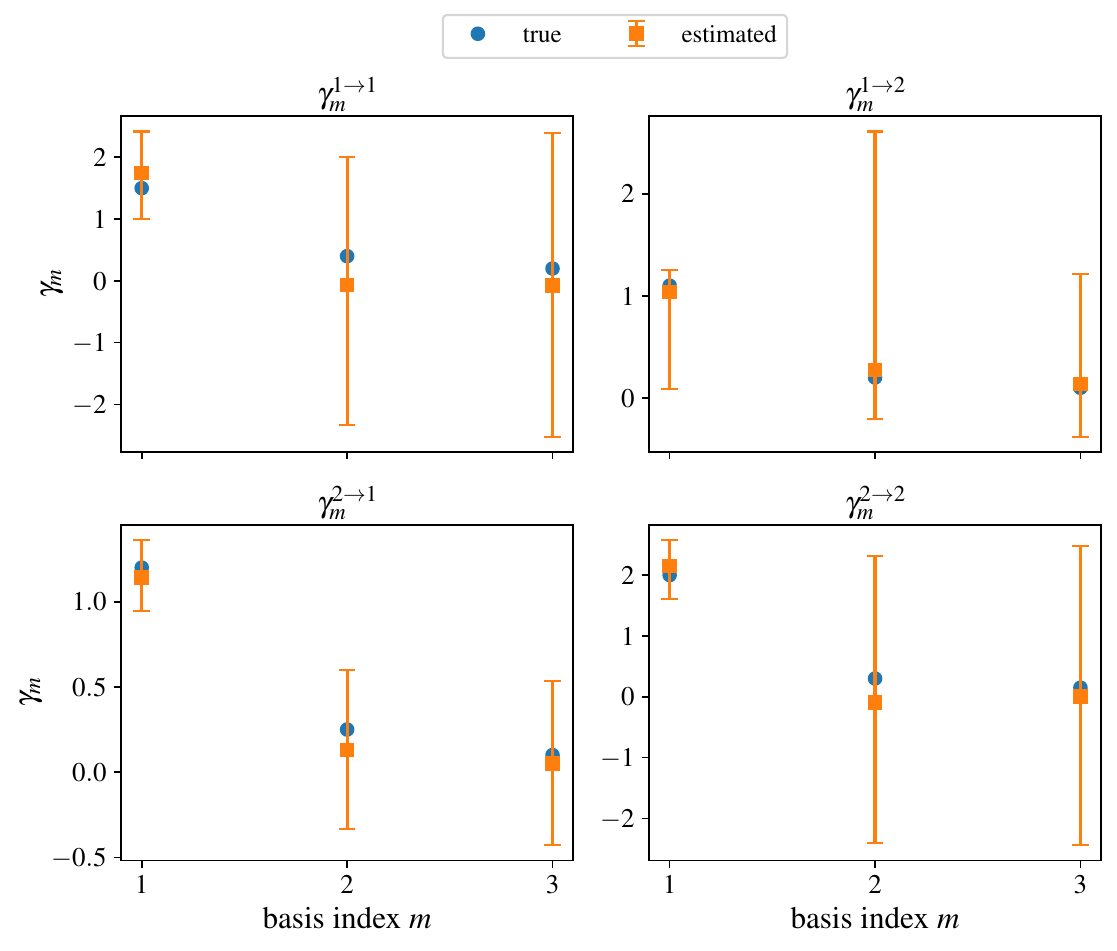}
    \caption{Recovery of basis coefficients $\{\gamma_m^{\ell \to h}\}$ by community pair. Marker locations denote the selected point estimate obtained from the median over selected runs, vertical error bars show the empirical 5th–95th percentile range across valid random-start runs, and circles indicate the true coefficients.}
    \label{fig:inferred_gamma_median}
\end{figure}

\begin{figure}
    \centering
    \includegraphics[width=0.75\linewidth]{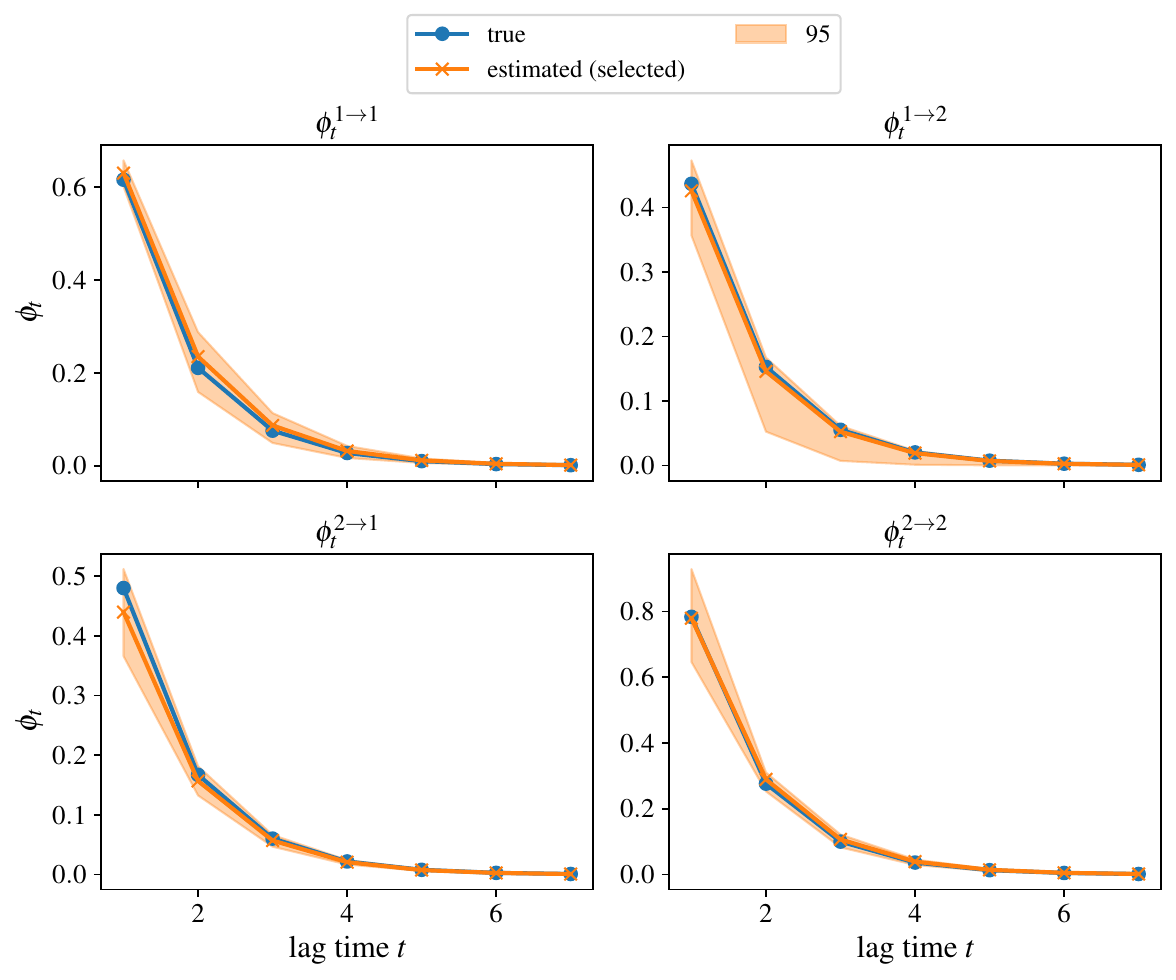}
    \caption{Reconstructed kernels $\{\phi_t^{\ell \to h}\}$ for each community pair under the exponential-basis model. The estimated curve is obtained by reconstructing the kernel from the selected median-based coefficient estimate; the shaded region shows the empirical across-run envelope (5th–95th percentiles), and the solid reference curve shows the true kernel.}
    \label{fig:inferred_kernel_median}
\end{figure}

\subsection{Complexity analysis}
A single gradient-ascent step with respect to the kernel parameters $\{\gamma_n^{s\to r}\}$ is computationally much less expensive under the grouped formulation induced by the mean-field approximation than under the full individual-level model. In the individual-level case, one must evaluate all intensities and their derivatives across $N$ units and $T$ time points, leading to complexity of order $\mathcal{O}(T^2N^2)$ under a naive implementation, or $\mathcal{O}(TN^2)$ when the exponential structure is exploited recursively. By contrast, the grouped formulation operates over only $q$ classes, where $q\ll N$ is fixed. Therefore the corresponding evaluations of the likelihood and its gradients are much cheaper; after an initial aggregation cost of order $\mathcal{O}(TN)$, the remaining updates depend only on class-level quantities and scale polynomially in $q$ and $p$ rather than $N$. This yields a substantial computational advantage, making the grouped mean-field approximation at most $\mathcal{O}(T^2N)$ for the naive implementation, and hence far more scalable than the full model.

\begin{Remark}
An alternative to the inference procedure described above would be to follow an approach similar to \cite{bacry16}, whereby the stochastic intensity is retained and the likelihood is approximated by an exactly tractable form using a Taylor expansion around the mean-field intensity. However, constructing a workable method along these lines is non-trivial and, from our investigations, unlikely to offer improved performance. For this reason, we do not consider it further here.
\end{Remark}

\section{Discussion} Section \ref{sec:application} acts as a proof of concept, demonstrating how the mean-field results can enable scalable inference in large-network settings. Extending this approach to more complex tasks -- such as clustering nodes when community labels are unavailable -- could make previously impractical approaches feasible by overcoming their scalability limitations.

Identifiability is a key concern in large-network models of this type and may be further exacerbated by the mean-field approximation. Replacing individual counts with their community-wise sums entails a loss of information, which may manifest itself as a degradation in model identifiability. A detailed theoretical investigation of this trade-off would offer a valuable avenue for future research.

While the results of this paper are formulated for the discrete time Poisson autoregression model, they naturally carry over to its continuous time analog, the Hawkes process. This is discussed in depth in Appendix \ref{sec:analogy}. 

\vspace{0.5cm}

\textbf{Acknowledgement:} MK wishes to thank Alexander Modell and Zoé
Agathe-Nerine for insightful discussions on the Stochastic Block Model and Hawkes processes on large graphs.\\ 
\textbf{Funding:} MK and EAKC acknowledge support from EPSRC NeST Programme grant EP/X002195/1. JM and JC acknowledge support from EPSRC StatML CDT grants EP/Y034813/1 and EP/S023151/1, respectively. \\

\newpage 

{\centering \Large{Supplementary Material}}

\appendix

\section{Analogy with Hawkes processes}
\label{sec:analogy}
As pointed out in the Introduction, Poisson autoregressions can be seen as discrete-time approximation of the continuous-time Hawkes process \cite{mahmoud2023, coutin2024functionalapproximationmarkedhawkes}. The convergence results thus carry over naturally to the Hawkes setting. The main difference in the continuous-time setting is that our driving randomness is no longer an \textit{iid} family of unit rate uni-dimensional unit-rate Poisson processes $\Pi= \left(\Pi^i_t(\cdot)\right)_{(i,t)\in \N ^*\times \N }$, but an \textit{iid} family of bi-dimensional Poisson measures $\Gamma = \left( \Gamma^i(\cdot,\cdot)_{i \in \N^*}\right)$ over $\R_+^2$ with a compensator $\mathrm d t \times \mathrm d \theta$. The Hawkes process $(H_t)_{t \in \R_+}$, as well as its stochastic intensity $(\lambda_t)_{t \in \R_+}$ are constructed on the SBM graph by "thinning" from $\Gamma$ according to the stochastic differential equation 
\begin{equation}
    \label{eq:thinning}
    \begin{cases}
        H_t^{i,N}&=\int_0^t\int_{\R_+}\mathds 1 _{\theta \leq \lambda ^{i,N}_s} \Gamma^{i}(\mathrm d s,\mathrm d \theta)\\
        \lambda^{i,N}_t &=\psi \left (\mu ^{C(i)}+\frac{1}{N}\sum_{j=1}^N \xi^{j \to i}\int_0^{t-}\int_{\R_+}\phi^{C(j)\to C(i)}(t-s) \mathds 1_{\theta \le \lambda^{j,N}_s} \Gamma ^{j}(\mathrm d s , \mathrm d \theta)\right).
    \end{cases}
\end{equation}
We refer the reader to \cite{bm, coutin2024functionalapproximationmarkedhawkes} for the construction of Hawkes processes via thinning. Equation \eqref{eq:thinning} is nothing but the continuous-time analog of Equation \eqref{eq: def_pois_auto}. 
The count series $Y^N$ can be seen as $\mathrm d H^N$, and the Hawkes process can simply be considered as a point process with stochastic intensity $\lambda_t ^{i,N} \mathrm d t = \mathbb P \left[\mathrm d H^{i,N}_t=1 \big | (H_s)_{s <t}\right],$ with $\lambda^N$ following the dynamics 
$$\lambda_t^{i,N}=\psi \left(\mu^{C(i)}+\frac{1}{N}\sum_{j=1}^N \xi ^{j \to i }\int_{0}^{t-}\phi^{C(j)\to C(i)}(t-s)\mathrm d H^{j,N}_s\right).$$
Just as for the discrete-time setting, we introduce the intermediary intensity $\hat \lambda^{N}=\left(\hat \lambda^{i,N}_t\right)_{(i,t)\in \{0,\cdots,N\}\times \R_+}$ constructed on the quenched graph as the solution of the equation 

\begin{align}
    \hat \lambda^{i,N}_{t}&:=\psi \left(\mu_{C(i)}+\frac{1}{N} \sum_{j=1}^N \xi_{j\to i} \int_{0}^t \phi ^{C(j)\to C(i)}(t-s)\hat{\lambda}^{j,N}_s \mathrm d s\right)\label{def:lambda_hat_ct} \\
    &=\psi \left(\mu_{C(i)}+\sum_{\ell=1}^q \alpha_\ell \frac{1}{N_\ell}\sum_{j: C(j)=\ell}  \xi_{j\to i} \int_0^t \phi ^{\ell\to C(i)}(t-s)\hat{\lambda}^{j,N}_s \mathrm d s\right) \nonumber.
\end{align}
The existence and uniqueness of $\hat \lambda ^N$ as the solution of Equation \eqref{def:lambda_hat_ct} (which is the continuous-time analog of Equation \eqref{def:lambda_hat}) can be proven using contraction arguments as in \cite{coutin2024functionalapproximationmarkedhawkes} or following the theorems of Chapter 6 in \cite{volterra}. The same can be applied to the fully deterministic intensity $\bar \lambda =\left(\bar \lambda ^\ell _t\right)_{(\ell,t)\in \{1,\cdots,q\}\times \R_+}$ defined as the unique solution of the equation
\begin{equation}
    \label{def: lambda_bar_ct}
    \bar \lambda ^\ell _t = \psi \left(\mu ^\ell+ \sum_{h=1}^q \alpha _h p^{h \to \ell } \int_0^t \phi^{h \to \ell  }(t-s) \bar \lambda ^h _s \mathrm ds\right),
\end{equation}
which is the continuous-time analog of Equation \eqref{def: lambda_bar}. Using the same broadcasting formula in \eqref{eq:bar_broadcast}, we define $\left(\bar \lambda ^{i,N}_t\right)_{(i,t)\in \{1,\cdots,N\}\times \R_+}$ as well as the (a priori inhomogeneous) Poisson process 
$$\bar H^{i,N}_t=\int_0^t \int_{\R_+} \mathds 1 _{\theta \le \bar \lambda ^{i,N}_s}\Gamma ^i(\mathrm d s , \mathrm d \theta ).$$
The mean-field approximation in continuous time is roughly  $H^{i,N}_t\simeq\bar H^{i,N}_t$ for all $t \in [0,T]$. Before we state the exact convergence result, we point out that the Euclidean stability condition in continuous time is written as
\begin{equation}
    \label{ass: ct}
    L\max_{h\in \{1,\cdots,q\}} \alpha_h \left\| \left(p_{h \to \ell}\int_0^{+\infty} \left|\phi^{h \to \ell}(t)\right |\mathrm d t\right)_{(\ell,h)\in \{1,\cdots,q\}^2}\right\|_2<1,
\end{equation}
which is clearly the continuous-time version of Assumption \ref{ass: crucial}. 
\begin{Theorem}
    Let $N \in \N^*$ be a fixed number of nodes and $T\in \R_+^*$ be the time horizon. If Assumption \ref{ass: ct} holds, then there exists two positive constants $C$ and $c$ that do not depend on $N$ nor $T$ such that, on the event $E^N$ (with $\Pb (E^N)\ge 1-2\mathrm e ^{-N}$) we have 
    $$\frac{1}{N} \sum_{i=1}^N\frac{1}{T}\E \left [\int _0^T \left | \mathrm d H^{i,N}_t-\mathrm d \bar H^{i,N}_t  \right |\right] \le \frac{C}{\sqrt N}\exp \left(c \frac{T}{\sqrt N}\right).$$
\end{Theorem}
\begin{proof}
    To keep this work at a reasonable length, we only give a partial proof of this theorem. For $i \in \{1,\cdots,N\}$, we have that 
    \begin{align*}
        \E \left [\int _0^T \left | \mathrm d H^{i,N}_t-\mathrm d \bar H^{i,N}_t  \right |\right]=&\E \left [\int _0^T \left | \int_{\R_+}(\mathds 1_{\theta \le \lambda^{i,N}_t}-\mathds 1_{\theta \le \bar\lambda^{i,N}_t}) \Gamma^i(\mathrm d t,\mathrm d \theta )  \right |\right]\\
        \le&\E \left [\int _0^T  \int_{\R_+}\left |\mathds 1_{\theta \le \lambda^{i,N}_t}-\mathds 1_{\theta \le \bar\lambda^{i,N}_t}\right |\Gamma^i(\mathrm d t,\mathrm d \theta )  \right]\\
        =& \E \left[\int_0^T \left|\lambda ^{i,N}_t-\bar \lambda^{i,N}_t \right | \mathrm d t\right].
    \end{align*}
    Therefore, using the Cauchy-Schwarz inequality we have that 
    \begin{align*}
        \frac{1}{N} \sum_{i=1}^N\frac{1}{T}\E \left [\int _0^T \left | \mathrm d H^{i,N}_t-\mathrm d \bar H^{i,N}_t  \right |\right] \le &\frac{1}{N} \sum_{i=1}^N\frac{1}{T}\E \left[\int_0^T \left|\lambda ^{i,N}_t-\bar \lambda^{i,N}_t \right | \mathrm d t\right]\\
        \leq & \frac{1}{\sqrt N} \sqrt{\sum_{i=1}^N \left(\frac{1}{T}\E \left[\int_0^T \left| \lambda^{i,N}_t-\bar \lambda^{i,N}_t\right| \mathrm d t\right]\right)^2}\\
        =&\frac{1}{\sqrt N} \sqrt{\left \|\frac{1}{T}\E \left[\int_0^T \left| \lambda^{N}_t-\bar \lambda^{N}_t\right| \mathrm d t\right] \right \|_2}.
    \end{align*}
    We then proceed to bound the difference between the intensities from above alongside Proposition \ref{prop:lambda_hat} and Proposition \ref{prop:hat_bar}.
\end{proof}

\section{Technical lemmata}
\label{App:lemmas}
\begin{Lemma}
    \label{lmm:convolution}
    Let $(A_k)_{k\in \N}$ be a family of non-negative matrices satisfying $A_0=0$ and $\rho \left (\mathcal S(A)\right)<1$. If a vector sequence $(x_t)_{t\in \N}$ satisfies 
    \begin{equation}
    \label{ineq:conv}
        |x_{t}|\preceq \sum_{k=0}^{t}A_{t-k} |x_{k}| + W_t,
    \end{equation}  
    for any $t\in \N$ and for some vector $(W_t)_{t \in \N}$, the
    $$|x_t| \preceq \left (\left(\sum_{m=0}^{+\infty}  A ^{(*m)} \right)*W\right)_t.$$
    In particular, if there exists a nonnegative constant vector sequence $\bar{W}$ such as $W_t \preceq \bar W$ then
    $$|x_t|_{ }\preceq \left (I-\sum_{k\geq 1}A_k \right)^{-1}\bar W.$$
\end{Lemma}
\begin{proof}
    Under a compact form, Inequality \eqref{ineq:conv} writes 
    $$|x_t| \preceq \left(A*|x|\right)_t+ W_t.$$
    By taking the convolution on the right and the left with $A^{(*m)}$ for some $m \in \N$ we have that 
    $$\left(A^{(*m)}*|x|\right)_t \preceq \left(A^{(*(m+1))}*|x|\right)_t+ \left(A^{(*m)}*W\right)_t,$$
    where the order has been preserved because taking the convolution involves adding and multiplying by positive matrices. Taking the telescopic sum of the last inequality yields 
$$|x_t|-\left(A^{(*(r+1))}*|x|\right)_t \preceq \left(\sum_{m=0}^{r}A^{(*m)}*W\right)_t,$$
for any $r \in \N$. For any $t \in \N$, we have that 
\begin{align*}
    \sum_{k=0}^t \left(A^{(*(r+1))}*|x|\right)_k &=\sum_{k=0}^t \sum_{j=0}^kA^{(*(r+1))}_j|x|_{k-j}\\
    &\preceq \left(\sum_{j=0}^{+\infty}A^{(*(r+1))}_j \right) \left(\sum_{k=0}^t|x|_k\right)\\
    &\preceq \left(\sum_{j=0}^{+\infty}A_j\right)^{r+1}\left(\sum_{k=0}^t|x|_k\right),
\end{align*}
and since $\rho \left(\sum_{j=0}^{+\infty}A_j\right)<1$, the right hand side of the last inequality converges to zero as $r$ grows to infinity, hence the result. 
\end{proof}
\begin{Lemma}
    \label{lmm: euclidean_norm}
    Let $\Phi^N$ be the kernel random matrix defined in Equation \eqref{def: Phi}. If Assumption \ref{ass: crucial} is in force, then we have that 
    $$\left \|\mathcal S\left(\Eb \left[\Phi^N\right]\right) \right\|_2 <1.$$
    And if Assumption \ref{ass:crucial_strong} holds, then 
    $$\mathcal S\left(\left \| \Eb \left[\Phi^N\right] \right\|_2\right) <1.$$
    Similarly, if Assumption \ref{ass:uniform} holds, then
    $$\left \|\mathcal S\left(\Eb \left[\Phi^N\right]\right) \right\|_\infty <1,$$
    and if Assumption \ref{ass:uniform_strong} holds then,
    $$\mathcal S\left(\left \| \Eb \left[\Phi^N\right] \right\|_\infty\right) <1.$$
\end{Lemma}
\begin{proof}
    First, we have that 
    $$\mathcal S\left(\Eb \left[\Phi^N\right]\right)^{ij}=\frac 1 N Lp^{C(j)\to C(i)}\mathcal S \left(\left|\phi^{C(j) \to C(i)} \right |\right) .$$
    By introducing the rectangular matrix $\Theta \in \mathcal M_{N\times q} (\R)$ defined by $\Theta_{ih}=\mathds 1 _{C(i)=h}$ and the square matrix $M' \in \mathcal M _{q \times q} (\R)$ given by $M'_{h\ell}=p^{\ell \to h} \mathcal S \left(\left |\phi^{\ell \to h}\right |\right)$, we clearly have that 
    $$\mathcal S\left(\Eb \left[\Phi^N\right]\right) = \frac 1 N L \Theta M' \Theta ^\top .$$
    Using the sub-multiplicativity of the Euclidean norm and the fact that $\|\Theta \|_2=\sqrt {\rho \left(\Theta^\top \Theta\right)}=\sqrt {N \max_{h \in \{1,\cdots,q\}} \alpha_h}$ we have that  
    $$\left \| \mathcal S\left(\Eb \left[\Phi^N\right]\right)\right \|_2 \leq L \max_{h \in \{1,\cdots,q\}} \alpha_h \|M'\|_2$$
    and we conclude regarding the first point using Assumption \ref{ass: crucial}.\\
    For the second point we simply have that 
    \begin{align*}
        \left \| \mathcal S\left(\Eb \left[\Phi^N\right]\right)\right \|_\infty =& L \max_{1 \leq i \le N}\sum_{j=1}^N\frac 1 N p^{C(j)\to C(i)} \mathcal S \left( \left| \phi^{C(j)\to C(i)}\right |\right)\\
        =& L \max_{1 \leq h \le q}\sum_{\ell=1}^q \frac 1 N\sum_{j:C(j)=\ell} p^{\ell\to h} \mathcal S \left( \left| \phi^{\ell \to h}\right |\right)\\
         =& L \max_{1 \leq h \le q}\sum_{\ell=1}^q \alpha^\ell  p^{\ell\to h} \mathcal S \left( \left| \phi^{\ell \to h}\right |\right),
    \end{align*}
    which yields the result.
\end{proof}

\begin{Lemma}
\label{lmm:average_adjacency}
    Let $N, T \in \N^*$ be two fixed integers. We recall that for $t=1,\cdots,T$ the random matrix $\Phi^N_t$ is given in Equation \eqref{def: Phi} by
    $$\left(\Phi^{ij,N}_{t}\right)=\frac{1}{N}L \xi_{j \to i} \phi ^{C(j)\to C(i)}_t.$$
    There exist then a positive constant $C$ that does not depend on $T$ nor on $N$ such that on the event $E^N$ ($\Pb(E^N) \ge 1-2\mathrm e^{-N}$) we have 
    $$\left \|\Phi^N_t-\Eb \left[\Phi^N_t\right]\right \|_2\leq C \sum_{h,\ell=1}^q \left |\phi^{h \to \ell}_t \right |\frac{1}{\sqrt N}$$
    for all $t \in \{1,\cdots,T\}$.
\end{Lemma}
\begin{proof}
    We recall that $\Xi=\left(\xi_{j \to i}\right)_{i,j=1,\cdots,N}$ is the adjacency matrix and we let $A_t$ be the $q \times q$ matrix defined by
    $$\left(A_t\right)^{h \ell}:=L |\phi^{\ell \to h}_t| .$$
    We have that  
    $$\Phi^N_t-\Eb \left[\Phi^N_t\right]=\left(\frac{1}{N}\Theta A_t \Theta ^ \top \right) \circ \left(\Xi-\Eb \left[ \Xi\right]\right),$$
    where $\Theta_{ih}=\mathds 1 _{C(i)=h}$ for $i=1,\cdots,N$ and $h=1,\cdots,q$ and $\circ$ denotes the Hadamard (or component-wise) product.\\ Using Lemma 5.6.7 in \cite{HJ} we have that 
    $$\left \| \Phi^N_t-\Eb \left[\Phi^N_t\right] \right \|_2 \leq \frac{1}{N}\sqrt{\| I \circ \Theta A_tA_t^\top \Theta ^\top \|_2 \left \| I \circ \Theta \Theta ^\top \right \|_2} \left \| \Xi -\Eb \left[ \Xi\right]\right \|_2.$$
   A straightforward computation yields that $\left \| I \circ \Theta \Theta ^\top \right \|_2=1$ and $$\| I \circ \Theta A_tA_t^\top \Theta ^\top \|_2 \leq L^2\sum_{h,\ell=1}^q\left | \phi^{\ell \to h}_t\right |^2.$$ Since all norms are equivalent in $\R^{q \times q}$, we have that 
   $$\sqrt{\| I \circ \Theta A_tA_t^\top \Theta ^\top \|_2} \leq C \sum_{h,\ell =1}^q \left | \phi^{h \to \ell}_t\right |,$$
   for some constant $C$ that only depends on $q$. 
   We then conclude using Lemma \ref{lmm:vershynin}.
\end{proof}

\section{Proofs}
\label{App:proofs}
\subsection{Proof of Proposition \ref{prop:cvg_bar_k}}
\label{proof: cvg_bar_k}
 Let $f$ be the function from $\R_+^q$ to itself defined as $$f^h (X)=\psi \left(\mu_h + \sum_{\ell=1}^q \alpha_\ell p_{\ell \to h}\left(\sum_{t=1}^{+\infty}\phi_{t}^{\ell \to h}\right) X^\ell \right ).$$ Using the fact that $\psi$ is $L-$Lipschitz, we have or any $X$ and $Z$ in $\R_+^q$
    \begin{align*}
        |f(X)-f(Z)| \preceq \bar M |X-Z|,
    \end{align*}
    where $\bar M$ is given in Assumption \ref{ass:stability_MF}. Given $\varepsilon >0$ such that $\rho(\bar M)+\varepsilon <1$, there exists a matrix norm $\|\cdot\|$ such that 
    \begin{align*}
        \|f(X)-f(Z)\| \leq (\rho(\bar M)+\varepsilon)\|X-Z\|,
    \end{align*}
    which means that $f$ is a contraction on a Banach space, hence the existence and uniqueness. Regarding the convergence, we clearly have for any $h \in \{1,\cdots,q\}$
    $$ \lambda^h_{\infty}=\psi \left( \mu_h +\sum_{\ell =1}^q \alpha_\ell p_{\ell \to h} \left(\sum_{k=1}^t \phi^{\ell \to h}_{t+1-k}  \lambda ^\ell _\infty +\sum_{k=t+1}^{+\infty} \phi^{\ell \to h}_{k}  \lambda ^\ell _\infty\right) \right)$$
    which combined with \eqref{def: lambda_bar} and the fact that $\psi$ is $L$-Lipschitz yields
    \begin{align*}
        |\bar \lambda_{t+1}^h-\lambda^h_\infty| &\leq L \left|\sum_{\ell =1}^q \alpha_\ell p_{\ell \to h}\left(\sum_{k=1}^{t} \phi^{\ell \to h}_{t+1-k} (\bar \lambda^\ell _k -\lambda _\infty ^\ell )-\sum_{k=t+1}^{+\infty} \phi^{\ell \to h}_{k}  \lambda ^\ell _\infty\right)\right| \\
        &\leq \sum_{k=1}^t\sum_{\ell=1}^q L \alpha_\ell p_{\ell \to h} |\phi^{\ell \to h}_{t+1-k}||\bar \lambda^\ell _k -\lambda _\infty ^\ell|+Lq \sum_{k=t+1}^{+\infty} \tilde \phi_{k} \lambda _\infty
    \end{align*}

    where we used $\max(\alpha _{\ell}p_{\ell \to h})\leq 1$, $\lambda_\infty =\max_{\ell =1,\cdots,q} \lambda_\infty^\ell$ and $\tilde \phi_k =\max_{\ell,h=1,\cdots,q}|\phi_k^{\ell \to h}|.$
    According to Lemma \ref{lmm:convolution} there exists a matrix sequence $B \in \ell_1(\N)$ such that 
    $$\left |\bar \lambda_{t+1}^h-\lambda^h_\infty \right |\preceq \left (B*\sum_{k=\cdot}^{+\infty}\tilde \phi_k\right)_{t+1}.$$
    Using the fact that $\lim_{t\to +\infty}\sum_{k=t+1}^{+\infty}\tilde \phi_k=0$ we conclude. 

    \subsection{Proof of Lemma \ref{lemma:hat_lambda}}
    \label{proof: hat_lambda}
        Using the Lipschitz property of $\psi$, we have that 
    \begin{align*}
      \hat\lambda^{i,N}_{t+1}&\leq \psi \left(\mu_{C(i)}\right) + L \sum_{\ell=1}^q \alpha_\ell \frac{1}{N_\ell}\sum_{j: C(j)=\ell}  \xi_{j\to i} \sum_{k=1}^t \left| \phi ^{\ell\to C(i)}_{t+1-k}\right |\hat{\lambda}^{j,N}_k\\
      &=\psi \left(\mu_{C(i)}\right) + \sum_{k=1}^t\left(\Phi^N_{t+1-k}\hat \lambda ^{N}_k\right)^i,
    \end{align*}
    which under matrix form writes
    \begin{align*}
\hat\lambda^{N}_{t+1}\preceq &\psi\left(\mu_{C(\cdot)}\right) + \sum_{k=1}^t\Phi^N_{t+1-k}\hat \lambda ^{N}_k\\
        =&\psi\left(\mu_{C(\cdot)}\right) + \sum_{k=1}^t\Eb\left[\Phi^N_{t+1-k}\right]\hat \lambda ^{N}_k+ \sum_{k=1}^t\left(\Phi^N_{t+1-k}-\Eb\left[\Phi^N_{t+1-k}\right]\right)\hat \lambda ^{N}_k.\\
    \end{align*}
    Since Lemma \ref{lmm: euclidean_norm} yields $\rho\left(\mathcal S \left(\Eb \left[\Phi^N \right]\right)\right) \leq  \left \| \mathcal S \left(\Eb \left[\Phi^N\right]\right)\right \|_2<1$  we can define the matrix sequence 
    $$R_n=\sum_{m=1}^{+\infty} \left(\Eb\left[\Phi^N\right]\right)^{(*m)}_n$$
    which satisfies 
    $$\mathcal S (R) = \sum_{m=1}^{+\infty}\sum_{n=0}^{+\infty} (\Eb\left[\Phi^N\right]^{*m})_n=\mathcal S \left(\Eb \left[\Phi^N\right]\right)\left(I_N-\mathcal S \left(\Eb \left[\Phi^N\right]\right)\right)^{-1}$$
    and using Lemma \ref{lmm:convolution} we have that
    $$\hat \lambda^{N}_{t+1}\preceq K_{t+1}+\left(R*K\right)_{t+1}$$
    where 
    \begin{align*}K_{t+1}&= \psi\left(\mu_{C(\cdot)}\right)+\sum_{k=1}^t\left(\Phi^N_{t+1-k}-\Eb\left[\Phi^N_{t+1-k}\right]\right)\hat \lambda ^{N}_k\\
    &=\psi\left(\mu_{C(\cdot)}\right)+\left(\left(\Phi^N-\Eb\left[\Phi^N\right]\right)*\hat \lambda^{N}\right)_{t+1}
    \end{align*}
    Hence, the auxiliary process satisfies 
    \begin{align*}
        \hat \lambda^{N}_{t+1} &\preceq \left((\delta_0^N+R)*\psi\left(\mu_{C(\cdot)}\right)\right)_{t+1}+\left(\left(\Phi^N-\Eb\left[\Phi^N\right]\right)*\hat \lambda^{N}\right)_{t+1} +\left(R*\left(\Phi^N-\Eb\left[\Phi^N\right]\right)*\hat \lambda^{N}\right)_{t+1}\\
        &=\left(I_N+\sum_{k=1}^{t+1}R_k\right)\cdot\psi\left(\mu_{C(\cdot)}\right)+\left(\left(\delta^N_0+R\right)*\left(\Phi^N-\Eb\left[\Phi^N\right]\right)*\hat \lambda^{N}\right)_{t+1}\\
        &\preceq \left(\sum_{m=0}^{+\infty}\left(\mathcal S \left(\Eb \left[\Phi^N \right]\right)\right)^m\right)\cdot\psi\left(\mu_{C(\cdot)}\right)+\left(\left(\delta^N_0+R\right)*\left(\Phi^N-\Eb\left[\Phi^N\right]\right)*\hat \lambda^{N}\right)_{t+1}\\
    \end{align*}

    where $\delta ^N_0$ is the matrix sequence equal to $I^N$ if $t=0$ and zero otherwise. By taking the $\|\cdot\|_2$ norm of the last inequality and using the triangle inequality we have that 
    \begin{align}
       \left \|\hat \lambda^{N}_{t+1} \right \|_2 \leq & \left(\sum_{m=0}^{+\infty}\|\mathcal S \left(\Eb \left[\Phi^N \right]\right)\|_2^m\right)\cdot \|\psi\left(\mu_{C(\cdot)}\right)\|_2 + \sum_{k=1}^{t}\left\|\left(\delta^N_0+R\right)*\left(\Phi^N-\Eb\left[\Phi^N\right]\right)_{t+1-k}\right\|_2\left\|\hat \lambda^{N}_k\right\|_2\nonumber \\
        \leq& \frac{\sqrt N}{1-\nu} \max_{h=1,\cdots,q}\psi(\mu_h) + \sum_{k=1}^{t}\left\|\left(\delta^N_0+R\right)*\left(\Phi^N-\Eb\left[\Phi^N\right]\right)_{t+1-k}\right\|_2\left\|\hat \lambda^{N}_k\right\|_2\label{ineq:proof_lemma_hat}
    \end{align}
    where $\nu= L\max_{h=1,\cdots,q}\alpha_h \|(p \mathcal S(|\phi|))\|_2$ (\textit{cf.} Assumption \ref{ass: crucial}) and 
    \begin{align*}
\left\|\left(\delta^N_0+R\right)*\left(\Phi^N-\Eb\left[\Phi^N\right]\right)_{t+1-k}\right\|_2 \leq & \sum_{j=1}^{t-k} \|(\delta^{N}_0+R)_{t+1-k-j}\|_2 \left\|\left(\Phi^N-\Eb\left[\Phi^N\right]\right)_j\right\|_2\\
        \leq & \sum_{j=1}^{t-k} \left\|\sum_{i=0}^{+\infty}(\delta^{N}_0+R)_{i}\right\|_2 \left\|\left(\Phi^N-\Eb\left[\Phi^N\right]\right)_j\right\|_2\\
        =& \frac{1}{1-\nu} \sum_{j=1}^{t-k} \left\|\left(\Phi^N-\Eb\left[\Phi^N\right]\right)_j\right\|_2.
    \end{align*}
    Using Lemma \ref{lmm:average_adjacency}  we have that on $E^N$ which has a probability larger than $1-2\mathrm e^{-N} $ we have that 
    $$\left\|\left(\Phi^N-\Eb\left[\Phi^N\right]\right)_j\right\|_2 \leq \frac {c}{\sqrt N} \sum_{h,\ell =1}^q \left | \phi ^{\ell \to h }_j\right |,$$
    for some $c>0$. Therefore, 
$$ \left \|\hat \lambda ^{N}_{t+1} \right \|_2 \le C \sqrt N + \sum_{k=1}^t\frac{c}{\sqrt N}\sum_{h,\ell=1}^q \mathcal S \left(\left | \phi^{\ell \to h }\right |\right)  \left \|\hat \lambda ^{N}_{k} \right \|_2.$$ 
    Then, using the discrete Gr\"onwall Lemma we have 
    $$ \left \|\hat \lambda^{N}_{t+1} \right \|_2 \leq C \sqrt N \exp \left(c \frac{T}{\sqrt N}\right),$$
    hence the first result. For the uniform upper bound, we assume that Assumption \ref{ass:crucial_strong} holds. This yields that the resolvent $R$ has a matrix norm bounded by 
    $$\|R_n\|_2\leq \sum_{m=1}^{+\infty} \left(\left\| \Eb \left[\Phi^N\right]\right\|_2\right)^{(*m)}_n,$$
    which ensures that 
    \begin{align*}
        \sum_{n=1}^{+\infty} \|R_n\|_2 \leq & \sum_{m=1}^{+\infty}\sum_{n=1}^{+\infty} \left(\left\| \Eb \left[\Phi^N\right]\right\|_2\right)^{(*m)}_n\\
        \leq &\sum_{m=1}^{+\infty} \left(\mathcal S \left(\left \| \Eb \left[\Phi^N\right] \right \|_2\right)\right)^m
    \end{align*}
    which by Lemma \ref{lmm: euclidean_norm} yields
    $$\sum_{n=1}^{+\infty} \|R_n\|_2 \leq C.$$
    Inequality \eqref{ineq:proof_lemma_hat} hence writes 
    \begin{align*}
        \left \| \hat \lambda ^{N}_{t+1}\right \|_2 \leq C \sqrt N+ \sum_{k=1}^tQ_{t+1-k} \left \| \hat \lambda ^{N}_{k} \right \|_2,
    \end{align*}
    where $Q_k =\sum_{j=1}^{k-1} \|(\delta^N_0+R)_{k-j}\|_2 \left \|\left(\Phi^N-\Eb\left[\Phi^N\right]\right)_j \right\|_2$. Using Lemma \ref{lmm:average_adjacency} we have that on $E^N$
    \begin{align*}\sum_{k=1}^{+\infty} Q_k \leq &\left(\sum_{j=1}^{+\infty}\|(\delta^N_0+R)_{j}\|_2\right) \left(\sum_{j=1}^{+\infty} \left \|\left(\Phi^N-\Eb\left[\Phi^N\right]\right)_j \right\|_2\right)\\
    \leq & \frac {C}{\sqrt{N}}\sum_{h,\ell=1}^q \mathcal S \left(\left |\phi^{\ell \to h} \right |\right).
    \end{align*}
    Hence, for $N$ larger than a certain threshold $N_0$ which does not depend on $T$ we have that $\sum_{k=1}^{+\infty} Q_k <1$ and we conclude by applyin Lemma \ref{lmm:convolution} a second time.  
    \subsection{Proof of Proposition \ref{prop:lambda_hat}}
    \label{proof:lambda_hat}
       For $i \in \{1,\cdots,N\}$, we start by bounding the difference between the intensity $\lambda^{i,N}$ and the auxiliary process $\hat \lambda^{i,N}$ for a realisation $\Xi$ of the SBM: 
    \begin{align*}
        \E\left[|\lambda^{i,N }_{t+1}-\hat \lambda ^{i,N}_{t+1}|\right] \leq & L \E \left |\sum_{\ell=1}^q \alpha_\ell \frac{1}{N_\ell}\sum_{j: C(j)=\ell}  \xi_{j\to i} \sum_{k=1}^t \phi ^{\ell\to C(i)}_{t+1-k}(Y^{j,N}_k-\hat{\lambda}^{j,N}_k)\right |\\
        \leq & L \E \left |\sum_{\ell=1}^q \alpha_\ell \frac{1}{N_\ell}\sum_{j: C(j)=\ell}  \xi_{j\to i} \sum_{k=1}^t \phi ^{\ell\to C(i)}_{t+1-k}\left(\Pi_k^{j}(\lambda^{j,N}_k)-\Pi_k^{j}(\hat{\lambda}^{j,N}_k)\right)\right |\\
        &+L \E \left |\sum_{\ell=1}^q \alpha_\ell \frac{1}{N_\ell}\sum_{j: C(j)=\ell}  \xi_{j\to i} \sum_{k=1}^t \phi ^{\ell\to C(i)}_{t+1-k}\left(\Pi_k^{j}(\hat{\lambda}^{j,N}_k)-\hat{\lambda}^{j,N}_k\right)\right |
    \end{align*}
    because $Y^{j,N}_k = \Pi_k^{j}({\lambda}^{j,N}_k)$. Since $\xi_{j\to i}$ for any $(i,j) \in \{1,\cdots, N\}^2$ is built on $(\Omega, \mathcal F_U, \Pb )$ we have that 
    \begin{align*}
        \E\left[|\lambda^{i,N }_{t+1}-\hat \lambda ^{i,N}_{t+1}|\right] \leq & L\sum_{\ell=1}^q \alpha_\ell \frac{1}{N_\ell}\sum_{j: C(j)=\ell}  \xi_{j\to i} \sum_{k=1}^t \left|\phi ^{\ell\to C(i)}_{t+1-k}\right |\E \left[\E_{k-1}\left |\Pi_k^{j}(\lambda^{j,N}_k)-\Pi_k^{j}(\hat{\lambda}^{j,N}_k)\right| \right]\\
        &+L\E\left|\sum_{j=1}^N \frac{1}{N} \xi_{j\to i} \sum_{k=1}^t \phi ^{C(j)\to C(i)}_{t+1-k} \left( \Pi_k^{j}(\hat{\lambda}^{j,N}_k)-\hat{\lambda}^{j,N}_k\right)\right|\\
        \leq &L\sum_{\ell=1}^q \alpha_\ell \frac{1}{N_\ell}\sum_{j: C(j)=\ell}  \xi_{j\to i} \sum_{k=1}^t \left|\phi ^{\ell\to C(i)}_{t+1-k}\right |\E \left |\lambda^{j,N}_k-\hat{\lambda}^{j,N}_k\right| \\
        &+L\sum_{k=1}^t \E\left| \frac{1}{N} \sum_{j=1}^N\xi_{j\to i}  \phi ^{C(j)\to C(i)}_{t+1-k} \left( \Pi_k^{j}(\hat{\lambda}^{j,N}_k)-\hat{\lambda}^{j,N}_k\right)\right|.\\
    \end{align*}

    
   Using the Cauchy Schwarz inequality and the fact that for every $k \in \{0,\cdots,T\}$, the $(\Pi^j_k)_{j\in \{1,\cdots, N\}}$ are independent, we have that 
   \begin{align*}
        \E\left[|\lambda^{i,N }_{t+1}-\hat \lambda ^{i,N}_{t+1}|\right] \leq & L\sum_{\ell=1}^q \alpha_\ell \frac{1}{N_\ell}\sum_{j: C(j)=\ell}  \xi_{j\to i} \sum_{k=1}^t \left|\phi ^{\ell\to C(i)}_{t+1-k}\right |\E \left |\lambda^{j,N}_k-\hat{\lambda}^{j,N}_k\right| \\
         &+L\sum_{k=1}^t  \frac{1}{N} \left[ \E\left| \sum_{j=1}^N\xi_{j\to i}  \phi ^{C(j)\to C(i)}_{t+1-k} \left( \Pi_k^{j}(\hat{\lambda}^{j,N}_k)-\hat{\lambda}^{j,N}_k\right)\right|^2 \right]^{1/2}\\
         \leq & L\sum_{\ell=1}^q \alpha_\ell \frac{1}{N_\ell}\sum_{j: C(j)=\ell}  \xi_{j\to i} \sum_{k=1}^t \left|\phi ^{\ell\to C(i)}_{t+1-k}\right |\E \left |\lambda^{j,N}_k-\hat{\lambda}^{j,N}_k\right| \\
         &+L\sum_{k=1}^t  \frac{1}{N} \left(  \sum_{j=1}^N\xi_{j\to i}  \left(\phi ^{C(j)\to C(i)}_{t+1-k} \right)^2 \mathrm{Var}\left( \Pi_k^{j}(\hat{\lambda}^{j,N}_k)\right) \right)^{1/2}\\
         \leq & L\sum_{\ell=1}^q \alpha_\ell \frac{1}{N_\ell}\sum_{j: C(j)=\ell}  \xi_{j\to i} \sum_{k=1}^t \left|\phi ^{\ell\to C(i)}_{t+1-k}\right |\E \left |\lambda^{j,N}_k-\hat{\lambda}^{j,N}_k\right| \\
         &+L\sum_{k=1}^t  \frac{1}{N} \left(  \sum_{j=1}^N  \left(\phi ^{C(j)\to C(i)}_{t+1-k} \right)^2 \hat{\lambda}^{j,N}_k \right)^{1/2}.
   \end{align*}
   Using Lemma \ref{lemma:hat_lambda} alongside the Cauchy-Schwarz inequality we have that, on the event $E^N$:
   \begin{align*}
       \sum_{k=1}^t  \frac{1}{N} \left(  \sum_{j=1}^N  \left(\phi ^{C(j)\to C(i)}_{t+1-k} \right)^2 \hat{\lambda}^{j,N}_k \right)^{1/2} &\leq \sum_{k=1}^t  \frac{1}{N} \left(  \left\|\hat{\lambda}^{N}_k\right\|_2\left(\sum_{j=1}^N  \left(\phi ^{C(j)\to C(i)}_{t+1-k} \right)^4\right)^{1/2}  \right)^{1/2}\\
       &=\sum_{k=1}^t  \frac{1}{N} \left(  \left\|\hat{\lambda}^{N}_k\right\|_2\sqrt N\left(\sum_{\ell=1}^q  \alpha_\ell\left(\phi ^{\ell\to C(i)}_{t+1-k} \right)^4\right)^{1/2}  \right)^{1/2}\\
       &\leq \sum_{k=1}^t  \frac{1}{N} \left(  C \sqrt N \exp \left(c\frac{T}{\sqrt N}\right)\sqrt N\left(\sum_{\ell=1}^q  \alpha_\ell\left(\phi ^{\ell\to C(i)}_{t+1-k} \right)^4\right)^{1/2}  \right)^{1/2}\\
       &= C  \exp \left(c\frac{T}{\sqrt N}\right)\frac{1}{\sqrt N}\sum_{k=1}^t    \left(\sum_{\ell=1}^q  \alpha_\ell\left(\phi ^{\ell\to C(i)}_{t+1-k} \right)^4\right)^{1/4}. 
   \end{align*}
   Since $ \sum_{k=1}^t\left(\sum_{\ell=1}^q  \alpha_\ell\left(\phi ^{\ell\to C(i)}_{t+1-k} \right)^4\right)^{1/4} \leq C$ independent of $t$ (yet potentially dependent on $q$), we get that 
   \begin{align}
   \label{ineq:lambda_hat}
        \E\left[|\lambda^{i,N }_{t+1}-\hat \lambda ^{i,N}_{t+1}|\right] \leq & C \frac{\exp\left(c\frac{T}{\sqrt N}\right)}{\sqrt N} +L\sum_{\ell=1}^q \alpha_\ell \frac{1}{N_\ell}\sum_{j: C(j)=\ell}  \xi_{j\to i} \sum_{k=1}^t \left|\phi ^{\ell\to C(i)}_{t+1-k}\right |\E \left |\lambda^{j,N}_k-\hat{\lambda}^{j,N}_k\right|, 
   \end{align}
   on $E^N$. 
   Following the exact same lines of the proof of Lemma \ref{lemma:hat_lambda} we have that 
   \begin{align*}
       \left\|\E \left[\left|\lambda_{t}^{N}-\hat \lambda_t^{N}\right|\right]\right\|_2\leq C \exp\left(c\frac{T}{\sqrt N}\right)
   \end{align*}
   on $E^N$ who has a probability $\Pb$ larger than $1-2\exp(-2N)$.
   \subsection{Proof of Proposition \ref{prop:hat_bar}}
   \label{proof:hat_bar}
  Under matrix format, we have that 
    \begin{align*}
        \hat \lambda^{N}_{t+1}=\psi\left(\mu_{C(\cdot)}+\sum_{k=1}^{t}\Phi^N_{t+1-k}\hat \lambda_k^{N}\right),
    \end{align*}
    and
    \begin{align*}
        \bar \lambda^{N}_{t+1}=\psi\left(\mu_{C(\cdot)}+\sum_{k=1}^t\Eb\left[\Phi^N_{t+k-1}\right]\bar \lambda^{N}_{k}\right).
    \end{align*}
    Hence, using the Lipschitz property of $\psi$, we have that
    \begin{align*}
        \left| \bar \lambda^{N}_{t+1}-  \hat \lambda^{N}_{t+1}\right| \preceq &  \left|\sum_{k=1}^t \Eb\left[\Phi^N_{t+1-k}\right]\bar \lambda^{N}_{k}-\Phi^N_{t+1-k}\hat \lambda_k^{N}\right|\\
        \preceq &\sum_{k=1}^t \Eb\left[\Phi^N_{t+1-k}\right]\left|\bar \lambda^{N}_{k}- \hat \lambda_k^{N}\right| +\left| \sum_{k=1}^t \left(\Eb\left[\Phi^N_{t+1-k}\right]-\Phi^N_{t+1-k}\right) \hat \lambda^{N}_k \right| \\
        =&\sum_{k=1}^t \Eb\left[\Phi^N_{t+1-k}\right]\left|\bar \lambda^{N}_{k}- \hat \lambda_k^{N}\right| +W_{t+1}
    \end{align*}
    where $W_{t+1}=\left| \sum_{k=1}^t \left(\Eb\left[\Phi^N_{t+1-k}\right]-\Phi^N_{t+1-k}\right) \hat \lambda^{N}_k \right|$. Hence, using Lemma \ref{lmm:convolution}, we have that 
    \begin{equation}
    \label{ineq:lambda_bar_hat}\left| \bar \lambda^{N}_{t+1}-  \hat \lambda^{N}_{t+1}\right| \preceq \left(\left(\delta^N_0+R\right)*W\right)_{t+1},
    \end{equation}
     with  $R_t=\sum_{m=1}^{+\infty}\left(\Eb \left[ \Phi^N\right]\right)^{(*m)}_t$.
    By taking the Euclidean norm of the last inequality, we have that 
    \begin{align}
        \left\| \bar \lambda^{N}_{t+1}-  \hat \lambda^{N}_{t+1}\right \|_2 \leq & \left \|\left(\left(\delta^N_0+R\right)*W\right)_t \right \|_2 \nonumber \\
        \leq & \sum_{k=1}^{t+1} \left \|\left(\delta^N_0+R\right)_{t+1-k} \right\|_2 \left \| W_k\right \|_2,\label{ineq:hat_bar}
    \end{align}
    where
    \begin{align*}
        \left \| W_{t+1}\right \|_2 &= \left \| \sum_{k=1}^t \left(\Eb\left[\Phi^N_{t+1-k}\right]-\Phi^N_{t+1-k}\right) \hat \lambda^{N}_k \right \|_2\\
        & \leq  \sum_{k=1}^t \left \| \left(\Eb\left[\Phi^N_{t+1-k}\right]-\Phi^N_{t+1-k}\right) \hat \lambda^{N}_k \right \|_2\\
        & \leq \sum_{k=1}^t \left \| \left(\Eb\left[\Phi^N_{t+1-k}\right]-\Phi^N_{t+1-k}\right) \right \|_2 \left \| \hat \lambda^{N}_k \right \|_2.\\
    \end{align*}
    Using Lemmas \ref{lemma:hat_lambda} and \ref{lmm:average_adjacency}, we have that on $E^N$ whose probability $\Pb$ is larger than $1-2\exp(-N)$ 
    \begin{align*} \left \| \left(\Eb\left[\Phi^N_{t+1-k}\right]-\Phi^N_{t+1-k}\right) \right \|_2 \left \| \hat \lambda^{N}_k \right \|_2 \leq &C \sum_{h,\ell=1}^q \left |\phi^{\ell \to h}_{t+1-k}\right|\frac{1}{\sqrt N} \sqrt N \exp \left(\frac{T}{\sqrt N}\right)\\
    \leq &C \sum_{h,\ell=1}^q \left |\phi^{\ell \to h}_{t+1-k}\right|\exp \left(\frac{T}{\sqrt N}\right),
    \end{align*}
    which by summation yields 
    \begin{align*}
        \|W_{t+1}\|_2& \leq \sum_{k=1}^t \left \| \left(\Eb\left[\Phi^N_{t+1-k}\right]-\Phi^N_{t+1-k}\right) \right \|_2 \left \| \hat \lambda^{N}_k \right \|_2\\
        & \leq C \sum_{k=1}^t\sum_{h,\ell=1}^q \left |\phi^{\ell \to h}_{t+1-k}\right|\exp \left(c\frac{T}{\sqrt N}\right)\\
        &\leq C \exp\left(c\frac{T}{\sqrt N}\right).
    \end{align*}
Hence Inequality \eqref{ineq:hat_bar} becomes 
\begin{align*} 
\left\| \bar \lambda^{N}_{t+1}-  \hat \lambda^{N}_{t+1}\right \|_2 & \leq C \exp \left(c \frac{T}{\sqrt N}\right) \sum_{k=1}^{t+1} \left \|\left(\delta^N_0+R\right)_{t+1-k} \right\|_2 \\
&\leq C \exp \left(c \frac{T}{\sqrt N}\right) \left(1+\sum_{k=1}^{t+1} \left \|\sum_{j=1}^{+\infty}R_{j} \right\|_2\right)\\
&\leq C \exp \left(c \frac{T}{\sqrt N}\right) \left(1+T\right)
\end{align*}
because $\|\mathcal S \left(\Eb \left[\Phi^N\right]\right)\|_2<1$ (\textit{cf.} Assumption \ref{ass: crucial} and Lemma \ref{lmm: euclidean_norm}). \\
 For the time averaged upper bound, we sum both sides of Inequality \eqref{ineq:lambda_bar_hat} for $t$ ranging from $1$ to $T$ to obtain 
    \begin{align*}
        \sum_{t=1}^T \left| \bar \lambda^{N}_{t}-  \hat \lambda^{N}_{t}\right|  \preceq & \left(\sum_{j=0}^{+\infty}(\delta^N_0+R)_j\right) \left(\sum_{k=1}^TW_k\right)\\
        =&\left(\sum_{m=0}^{+\infty}\mathcal S(\Eb \left[\Phi^N\right])^m\right) \left(\sum_{k=1}^TW_k\right).\\
    \end{align*}
    Taking the Euclidean norm of the last inequality yields
    \begin{align*}
        \left \| \sum_{t=1}^T \left| \bar \lambda^{N}_{t}-  \hat \lambda^{N}_{t}\right|  \right\|_2\leq & \left \|\left(\sum_{m=0}^{+\infty}\mathcal S(\Eb \left[\Phi^N\right])^m\right) \left(\sum_{k=1}^TW_k\right) \right \|_2\\
        \leq & \left(\sum_{m=0}^{+\infty}\left \|\mathcal S(\Eb \left[\Phi^N\right])\right \|^m_2\right)\left(\sum_{k=1}^T \left \|W_k \right \|_2 \right),
    \end{align*}
    thus 
    $$ \frac{1}{T} \left \| \sum_{t=1}^T \left| \bar \lambda^{N}_{t}-  \hat \lambda^{N}_{t}\right|  \right\|_2 \leq C \exp \left(c \frac{T}{\sqrt N}\right).$$

If Assumption \ref{ass:crucial_strong} holds, using the sub-multiplicativity of the Euclidean matrix norm we have that $\|W_k\|_2\le C$ for all $t\in \N$ and  $ \|R_t\|_2\le \sum_{m=1}^{+\infty}\left( \left \|\Eb \left[ \Phi^N\right] \right \|_2\right)^{(*m)}_t$, thus
\begin{align*}
\left\| \bar \lambda^{N}_{t+1}-  \hat \lambda^{N}_{t+1}\right \|_2 \leq &C  \sum_{k=1}^{t+1} \left \|\left(\delta^N_0+R\right)_{t+1-k} \right\|_2\\
\leq &C \left(1+\sum_{k=1}^{+\infty} \left \|R_{k} \right\|_2\right)
\end{align*}
and we conclude because $ \mathcal S\left(\| \left(\Eb \left[\Phi^N\right]\right)\|_2\right)<1$ as per Lemma \ref{lmm: euclidean_norm}.
    \subsection{Proof of Theorem \ref{thm:main_l_2}}
\label{proof:main_l_2}
    Using the Cauchy-Schwarz inequality we have that 
    $$\frac{1}{N}\sum _{i=1}^N \E \left[ \left| Y_t^{i,N}-\bar Y_t^{i,N}\right|\right] \leq \frac{1}{\sqrt N}\left(\sum _{i=1}^N \E \left[ \left| Y_t^{i,N}-\bar Y_t^{i,N}\right|\right]^2\right)^{1/2}$$
    and since $Y_t^{i,N}=\Pi_t \left(\lambda_t^{i,N}\right)$ and $\bar Y^{i,N}_t =\Pi _t\left(\bar \lambda_t^{i,N}\right)=\Pi _t\left(\bar \lambda_t^{C(i)}\right)  $ , 
    \begin{align*}
        \frac{1}{N}\sum _{i=1}^N \E \left[ \left| Y_t^{i,N}-\bar Y_t^{i,N}\right|\right] &\leq\frac{1}{\sqrt N}\left(\sum _{i=1}^N \E \left[ \left| \lambda_t^{i,N}-\bar \lambda ^{C(i)}_t\right|\right]^2\right)^{1/2}\\
        &=\frac{1}{\sqrt N}\left\| \E \left[ \left| \lambda_t^{N}-\bar \lambda ^{N}_t\right| \right]\right \|_2\\
        &\leq \frac{1}{\sqrt N}\left\| \E \left[ \left| \lambda_t^{N}-\hat \lambda ^{N}_t\right| \right]\right \|_2 +  \frac{1}{\sqrt N}\left\| \hat\lambda_t^{N}-\bar \lambda ^{N}_t\right \|_2
    \end{align*}
    which, using Propositions \ref{prop:hat_bar} and \ref{prop:lambda_hat} yields the result.\\
   \subsection{Proof of Proposition \ref{prop:union_bound}}
   \label{proof:union_bound}
For a given $t=1,\cdots,T$ and $k = 1,\cdots,t $ the definition of the infinity norm and the matrix product yield that

\begin{align*}
     \left \| \left (\Eb [\Phi^N_{t-k+1}]-\Phi^N_{t-k+1}\right) \bar \lambda^{N}_k \right \|_\infty &=\max_{i=1,\cdots,N} \left | \left(\left (\Eb [\Phi^N_{t-k+1}]-\Phi^N_{t-k+1}\right) \bar \lambda^{N}_k \right )_i \right |\\
     &= L\max_{i=1,\cdots,N} \left | \sum_{j=1}^N \frac{\xi^{j \to i}-p^{C(j)\to C(i)}}{N} \phi_{t-k+1}^{C(j)\to C(i)} \bar \lambda^{j,N}_k\right |\\
     & \leq L \max_{i=1,\cdots,N} \sum_{\ell=1}^q \left | \phi_{t-k+1}^{\ell \to C(i)} \bar \lambda^{\ell}_k \right | \left |\sum_{j: C(j)=\ell} \frac{\xi^{j \to i}-p^{\ell\to C(i)}}{N}\right |\\
     & \leq L \max_{h,\ell=1,\cdots,q}\left | \phi_{t-k+1}^{\ell \to h} \bar \lambda^{\ell}_k \right | \max_{i=1,\cdots,N} \sum_{\ell=1}^q  \left |\sum_{j: C(j)=\ell} \frac{\xi^{j \to i}-p^{\ell\to C(i)}}{N}\right |.\\
\end{align*}

Let $u>0$ be a threshold. Using the independence of the edges $\xi^{j\to i}$, we have that 
\begin{align*}
    \Pb \Bigg( \max_{i=1,\cdots,N} \sum_{\ell=1}^q \alpha_\ell &\left |\sum_{j: C(j)=\ell} \frac{\xi_{j \to i}-p_{\ell\to C(i)}}{N_\ell}\right | \leq u \Bigg)\\
    &= \prod_{i=1}^N \Pb \left(\sum_{\ell=1}^q \alpha_\ell  \left |\sum_{j: C(j)=\ell} \frac{\xi_{j \to i}-p_{\ell\to C(i)}}{N_\ell }\right | \leq \sum_{\ell=1}^q \alpha_\ell u\right)\\
    & \geq  \prod_{i=1}^N \Pb \left(  \left |\sum_{j: C(j)=\ell} \frac{\xi_{j \to i}-p_{\ell\to C(i)}}{N_\ell }\right | \leq  u,\quad \forall \ell=1,\cdots q\right)\\
    &= \prod_{i=1}^N \prod_{\ell=1}^q\Pb \left(  \left |\sum_{j: C(j)=\ell} \frac{\xi_{j \to i}-p_{\ell\to C(i)}}{N_\ell }\right | \leq  u\right).\\
\end{align*}
Hoeffding's concentration inequality yields
\begin{align*}
    \Pb \Bigg( \max_{i=1,\cdots,N} \sum_{\ell=1}^q \alpha_\ell \left |\sum_{j: C(j)=\ell} \frac{\xi_{j \to i}-p_{\ell\to C(i)}}{N_\ell}\right | \leq u \Bigg)& \geq \prod_{i=1}^N \prod_{\ell=1}^q \left(1-\Pb \left(  \left |\sum_{j: C(j)=\ell} \frac{\xi_{j \to i}-p_{\ell\to C(i)}}{N_\ell }\right | \geq  u\right)\right)\\
   & \geq \prod_{i=1}^N \prod_{\ell=1}^q \left(1- 2 \exp \left(-2u^2N_\ell\right)\right)\\
   &=\exp \left(\sum_{i=1}^N \sum_{\ell=1}^q\ln \left (1- 2 \exp \left(-2u^2N_\ell\right)\right) \right)\\
   &=\exp \left( N\sum_{\ell=1}^q\ln \left (1- 2 \exp \left(-2u^2N_\ell\right)\right) \right).
\end{align*}
Using the lower bound $\ln (1-x)\geq \frac{-x}{\sqrt{1-x}}$ for $x \in (0,1)$ we have 
\begin{align*}
     \Pb \Bigg( \max_{i=1,\cdots,N} \sum_{\ell=1}^q \alpha_\ell \left |\sum_{j: C(j)=\ell} \frac{\xi_{j \to i}-p_{\ell\to C(i)}}{N_\ell}\right | \leq u \Bigg)& \geq \exp \left( N\sum_{\ell=1}^q\ln \left (1- 2 \exp \left(-2u^2N_\ell\right)\right) \right)\\
     &=\exp \left( N\sum_{\ell=1}^q \frac{ -2 \exp \left(-2u^2N_\ell\right)}{\sqrt{1- 2 \exp \left(-2u^2N_\ell\right)}}\right).\\
\end{align*}
For a given $\varepsilon \in (0, 1/2)$, we let $u=N^{-1/2+\varepsilon}$, which yields 
\begin{align*}
  \Pb \Bigg( \max_{i=1,\cdots,N} \sum_{\ell=1}^q \alpha_\ell \left |\sum_{j: C(j)=\ell} \frac{\xi_{j \to i}-p_{\ell\to C(i)}}{N_\ell}\right | \leq \frac{N^\varepsilon}{\sqrt N} \Bigg)& \geq \exp \left( N\sum_{\ell=1}^q \frac{ -2 \exp \left(-2 \alpha_\ell N^{2\varepsilon}\right)}{\sqrt{1- 2 \exp \left(-2 \alpha_\ell N^{2\varepsilon}\right)}}\right).
\end{align*}
If $N\geq \left(\frac{\ln(8/3)}{2\alpha_{\min}}\right)^{\frac{1}{2\varepsilon}}$, then $\sqrt{1-2\exp(-2\alpha_\ell N^{2\varepsilon})}\geq \frac{1}{2}$ for all $\ell=1,\cdots,q$ and therefore 
\begin{align*}
    \Pb \Bigg( \max_{i=1,\cdots,N} \sum_{\ell=1}^q \alpha_\ell \left |\sum_{j: C(j)=\ell} \frac{\xi_{j \to i}-p_{\ell\to C(i)}}{N_\ell}\right | \leq \frac{N^\varepsilon}{\sqrt N} \Bigg)& \geq \exp \left( N\sum_{\ell=1}^q \frac{ -2 \exp \left(-2 \alpha_\ell N^{2\varepsilon}\right)}{\sqrt{1- 2 \exp \left(-2 \alpha_\ell N^{2\varepsilon}\right)}}\right) \\
    &\geq \exp \left( N\sum_{\ell=1}^q  -4 \exp \left(-2 \alpha_\ell N^{2\varepsilon}\right)\right)\\
    &\geq \exp \left(-4 Nq  \exp \left(-2 \alpha_{\max} N^{2\varepsilon}\right)\right)\\
    &\geq 1-4 Nq  \exp \left(-2 \alpha_{\max} N^{2\varepsilon}\right).
\end{align*}
Therefore, with a probability $\Pb$ larger than $ 1-4 Nq  \exp \left(-2 \alpha_{\max} N^{2\varepsilon}\right)$ we have that 
$$\left \| \left (\Eb [\Phi^N_{t-k+1}]-\Phi^N_{t-k+1}\right) \bar \lambda^{N}_k \right \|_\infty \leq L \max_{h,\ell=1,\cdots,q} \left | \phi^{\ell\to h}_{t-k+1} \bar \lambda^\ell_k\right| \frac{N^\varepsilon}{\sqrt N} $$
for any $t=1,\cdots,T$ and $k=1,\cdots,t$
and the result follows. 

\subsection{Proof of Proposition \ref{prop:intensity_uniform}}
\label{proof:intensity_uniform}
 First, we prove an analog of Lemma \ref{lemma:hat_lambda} for the uniform norm. Following the same lines as the proof of the aforementioned lemma we have that, 
    \begin{align*}\hat \lambda ^{N}_{t+1}\preceq &\left(\sum_{m=0}^{+\infty}\mathcal S\left(\Eb \left[\Phi^N\right]\right)^m\right)\cdot\psi\left(\mu_{C(\cdot)}\right)+\left(\left(\delta^N_0+R\right)*\left(\Phi^N-\Eb\left[\Phi^N\right]\right)*\hat \lambda^{N}\right)_{t+1}\\
    = &\left(\sum_{m=0}^{+\infty}\mathcal S\left(\Eb \left[\Phi^N\right]\right)^m\right)\cdot\psi\left(\mu_{C(\cdot)}\right)+\sum_{k=1}^{t} \left(\left(\delta^N_0+R\right)*\left(\Phi^N-\Eb\left[\Phi^N\right]\right)\right)_{t+1-k}\hat \lambda^{N}_k.\\
    \end{align*}
    Using the triangular inequality and the consistence of the norms we have that 
    \begin{align*}
       \left \|\hat \lambda^{N}_{t+1} \right \|_\infty \leq & \left(\sum_{m=0}^{+\infty}\|\mathcal S\left(\Eb \left[\Phi^N\right]\right)\|_\infty^m\right)\cdot \|\psi\left(\mu_{C(\cdot)}\right)\|_\infty + \sum_{k=1}^{t}\left\|\left(\delta^N_0+R\right)*\left(\Phi^N-\Eb\left[\Phi^N\right]\right)_{t-k}\right\|_{2\to \infty}\left\|\hat \lambda^{N}_k\right\|_2\\
        \leq& \frac{\max_{h=1,\cdots,q}\psi(\mu_h)}{1-\nu'}  + \sum_{k=1}^{t}\left\|\left(\delta^N_0+R\right)*\left(\Phi^N-\Eb\left[\Phi^N\right]\right)_{t-k}\right\|_{2\to \infty}\left\|\hat \lambda^{N}_k\right\|_2
    \end{align*}
    with $\nu'$ is given in Assumption \ref{ass:uniform}. Using the obvious inequality $\|AB\|_{2\to \infty}\leq \|A\|_\infty \|B\|_{2\to \infty}$,
\begin{align*}
        \left\|\left(\delta^N_0+R\right)*\left(\Phi^N-\Eb\left[\Phi^N\right]\right)_{t-k}\right\|_{2\to \infty} \leq & \sum_{j=1}^{t-k} \|(\delta^{N}_0+R)_{t-k-j}\|_\infty \left\|\left(\Phi^N-\Eb\left[\Phi^N\right]\right)_j\right\|_{2\to \infty}\\
        \leq & \sum_{j=1}^{t-k} \left\|\sum_{i=0}^{+\infty}(\delta^{N}_0+R)_{i}\right\|_\infty \left\|\left(\Phi^N-\Eb\left[\Phi^N\right]\right)_j\right\|_{2\to \infty}\\
        =& \frac{1}{1-\nu'} \sum_{j=1}^{t-k} \left\|\left(\Phi^N-\Eb\left[\Phi^N\right]\right)_j\right\|_{2\to \infty}.
\end{align*}
Since for any given matrix $A$ we have that $\|A\|_{2\to \infty}\leq \|A\|_2$ (\textit{cf.} \cite{Cape}) and using Lemma \ref{lmm:average_adjacency}, there exists an absolute constant $C>0$ such that 
\begin{align*}
    \left \|\left(\Phi^N-\Eb\left[\Phi^N\right]\right)_j \right \|_{2 \to \infty} \leq \frac{C}{\sqrt N} \sum_{h,\ell=1}^q \left | \phi^{\ell \to h}_j\right |
\end{align*}
on the event $E^N$ whose probability $\Pb$ is at least $1-2\exp(-N)$. We therefore have that 
\begin{align*}
\left \| \hat \lambda^{N}_{t+1}\right \|_\infty \leq & C+\sum_{k=1}^t \frac{C}{\sqrt N} \sum_{h,\ell=1}^q  \mathcal S(\phi^{\ell \to h}) \left\| \hat \lambda_k^{N}\right \|_2\\
\leq & C+\sum_{k=1}^t \frac{C}{\cancel{\sqrt N}} \sum_{h,\ell=1}^q  \mathcal S(\phi^{\ell \to h})\cancel{\sqrt N} \exp\left(c \frac{T}{\sqrt N}\right)
\end{align*}
where the last upper bound comes from Lemma \ref{lemma:hat_lambda}. Hence, 
$$\left \| \hat \lambda^{N}_{t+1}\right \|_\infty \leq CT\exp\left(c \frac{T}{\sqrt N}\right)$$
for $t \in \{0,\cdots,T\}$ on $E^N$.\\
We point out that if we work under the stronger Assumption \ref{ass:uniform_strong}, we have that 
$$\left \| \hat \lambda^{N}_{t+1}\right \|_\infty\le C$$
on $E^N$ for $t \in \N$. \\
 Under matrix form, Inequality \eqref{ineq:lambda_hat} writes
\begin{align*}
    \E \left[|\lambda^{N}_{t+1}-\hat \lambda ^{N}_{t+1}|\right]\preceq & C \frac{\exp\left(c \frac{T}{\sqrt N}\right)}{\sqrt {N}} \boldsymbol{1} + \sum_{k=1}^t \Phi^N_{t+1-k} \E \left[|\lambda^{N}_{k}-\hat \lambda ^{N}_{k}|\right]\\
    \preceq & C \frac{\exp\left(c \frac{T}{\sqrt N}\right)}{\sqrt {N}} \boldsymbol{1} + \sum_{k=1}^t \left( \Phi^N_{t+1-k} -\Eb \left[\Phi^N_{t+1-k}\right]\right)\E \left[|\lambda^{N}_{k}-\hat \lambda ^{N}_{k}|\right]\\
    &+\sum_{k=1}^t \Eb \left[\Phi^N_{t+1-k}\right] \E \left[|\lambda^{N}_{k}-\hat \lambda ^{N}_{k}|\right].
\end{align*}
which by virtue of Lemma \ref{lmm:convolution} yields 
\begin{align*}
    \E \left [\left |\lambda^{N}_{t+1}-\hat \lambda ^{N}_{t+1} \right |\right] \preceq &  C \frac{\exp\left(c \frac{T}{\sqrt N}\right)}{\sqrt {N}} \left(\sum_{m=0}^{+\infty} \mathcal S(\Eb [\Phi^N])^m \right) \cdot \boldsymbol{1}\\
    &+ \sum_{k=1}^t \left((\delta ^N_0+R)*(\Phi^N-\Eb[\Phi^N])\right)_{t+1-k}\cdot \E \left [\left |\lambda^{N}_{k}-\hat \lambda ^{N}_{k} \right |\right]
\end{align*}

Which after taking the infinity norm and using the triangle inequality yields 
\begin{align*}
      \left \|\E \left [\left |\lambda^{N}_{t+1}-\hat \lambda ^{N}_{t+1} \right |\right] \right \|_\infty\leq &  C \frac{\exp\left(c \frac{T}{\sqrt N}\right)}{\sqrt {N}} \left(\sum_{m=0}^{+\infty} \left \|  \mathcal S(\Eb[\Phi^N])\right \|_\infty^m \right) \\
    &+ \sum_{k=1}^t \left \|\left((\delta ^N_0+R)*(\Phi^N-\Eb[\Phi^N])\right)_{t-k} \right \|_{2 \to \infty}\left \| \E \left [\left |\lambda^{N}_{k}-\hat \lambda ^{N}_{k} \right |\right]\right \|_2
\end{align*}
Using Proposition \ref{prop:lambda_hat} we obtain that 
\begin{align*}
     \left \|\E \left [\left |\lambda^{N}_{t+1}-\hat \lambda ^{N}_{t+1} \right |\right] \right \|_\infty\leq &  C \frac{\exp\left(c \frac{T}{\sqrt N}\right)}{\sqrt {N}} \left(\sum_{m=0}^{+\infty} \left \|  \mathcal S(\Eb[\Phi^N])\right \|_\infty^m \right) + \sum_{k=1}^t \frac{C}{\sqrt N} \exp \left( c \frac{T}{\sqrt N}\right)\\
     \leq & \frac{ C(1+T)}{\sqrt N} \exp \left (c\frac{T}{\sqrt N}\right).
\end{align*}

We now bound the difference between the auxiliary process $\hat \lambda^{N}$ with frozen graph randomness and the mean-field deterministic intensity $\bar \lambda ^{N}$. Following the proof of Proposition \ref{prop:hat_bar}, Inequality  
\begin{equation*}
    \left| \bar \lambda^{N}_{t+1}-  \hat \lambda^{N}_{t+1}\right| \preceq \left(\left(\delta^N_0+R\right)*W\right)_t,
    \end{equation*}
    holds, where $W_t=\left| \sum_{k=1}^t \left(\Eb\left[\Phi^N_{t+1-k}\right]-\Phi^N_{t+1-k}\right) \hat \lambda^{N}_k \right|$ and $R_t=\sum_{m=1}^{+\infty}\left(\Eb \left[ \Phi^N\right]\right)^{(*m)}_t$. Taking the infinity norm yields
    \begin{align*}
        \left \| \bar \lambda^{N}_{t+1}-  \hat \lambda^{N}_{t+1}\right \|_\infty \leq \sum_{k=1}^t\left\|\left(\delta^N_0+R\right)_{t-k} \right \|_\infty \left \|W_k\right \|_\infty,
    \end{align*}
    with 
 \begin{align*}
        \|W_k\|_\infty& \leq \sum_{k=1}^t \left \| \left(\Eb\left[\Phi^N_{t+1-k}\right]-\Phi^N_{t+1-k}\right) \hat \lambda^{N}_k \right \|_{\infty}\\
        &\leq  \sum_{k=1}^t \left \| \left(\Eb\left[\Phi^N_{t+1-k}\right]-\Phi^N_{t+1-k}\right) \bar \lambda^{N}_k \right \|_{\infty}+\left \| \left(\Eb\left[\Phi^N_{t+1-k}\right]-\Phi^N_{t+1-k}\right) \left(\hat \lambda_k^{N}-\bar \lambda^{N}_k \right)\right \|_{\infty}\\
        &\leq  \sum_{k=1}^t \left \| \left(\Eb\left[\Phi^N_{t+1-k}\right]-\Phi^N_{t+1-k}\right) \bar \lambda^{N}_k \right \|_{\infty}+\left \| \Eb\left[\Phi^N_{t+1-k}\right]-\Phi^N_{t+1-k}\right \|_{2 \to \infty} \left \|\hat \lambda_k^{N}-\bar \lambda^{N}_k\right \|_{2}\\
        \end{align*}
        which according to Proposition \ref{prop:hat_bar} becomes
        $$\|W_k\|_\infty \leq \sum_{k=1}^t \left \| \left(\Eb\left[\Phi^N_{t+1-k}\right]-\Phi^N_{t+1-k}\right) \bar \lambda^{N}_k \right \|_{\infty} + \frac{CT \exp \left(c \frac{T}{\sqrt N}\right)}{\sqrt N} \sum_{h,\ell=1}^q \left \|\phi^{\ell \to h}\right \|_1.$$

                Using Proposition \ref{prop:union_bound} in conjunction with the fact that $ \left(\bar \lambda_k\right)_{k \in \N}$ converges in $k$ (Lemma \ref{prop:cvg_bar_k}) we have that 
        $$\|W_k\|_\infty \leq C \left(\frac{N^{\varepsilon}+T \exp \left(c\frac{T}{\sqrt N} \right)}{\sqrt N}\right),$$
        under the event $F^N \cap E^N$ whose probability $\Pb$ is larger than $1-4Nq \exp(-2\alpha_{\max}N^{2\varepsilon})-2\exp(-N)$. Hence,
        \begin{align*}
              \left \| \bar \lambda^{N}_{t+1}-  \hat \lambda^{N}_{t+1}\right \|_\infty \leq &\sum_{k=1}^t\left\|\left(\delta^N_0+R\right)_{t-k} \right \|_\infty \left \|W_k\right \|_\infty\\
              \leq & CT\left(\frac{N^{\varepsilon}+T \exp \left(c\frac{T}{\sqrt N} \right)}{\sqrt N}\right).
        \end{align*}
        And therefore, under Assumptions \ref{ass: crucial} and \ref{ass:uniform}, 

        $$\left \|\E \left [\left |\lambda^{N}_{t+1}-\bar \lambda ^{N}_{t+1} \right |\right] \right \|_\infty\leq CT\left(\frac{N^{\varepsilon}+T \exp \left(c\frac{T}{\sqrt N} \right)}{\sqrt N}\right)$$
        with a probability $\Pb$ larger than $1-4Nq \exp(-2\alpha_{\max}N^{2\varepsilon})-2\exp(-N)$. 

        If the stronger Assumptions \ref{ass:crucial_strong} and \ref{ass:uniform_strong} hold, then we can bound the matrix norms of the resolvents at any time and we obtain the bound that does not depend on $T$. 

\section{Additional simulation results}\label{sec:add_sims}
Additional figures to support Section \ref{sec:verification}.

\subsection{Intermediary results for dependence on $N$}
Figure \ref{fig:N_two_norms} demonstrates the intermediary convergence results in $N$ for the true intensity and the auxiliary intensity, and the auxiliary intensity and the mean-field intensity, in Eucliead norm. Specifically, they show that the dependencies remain approximately constant and bounded across network sizes, as shown in in Propositions \ref{prop:lambda_hat} and \ref{prop:hat_bar}.
\begin{figure}[h!]
    \centering
    \includegraphics[width=0.9\linewidth]{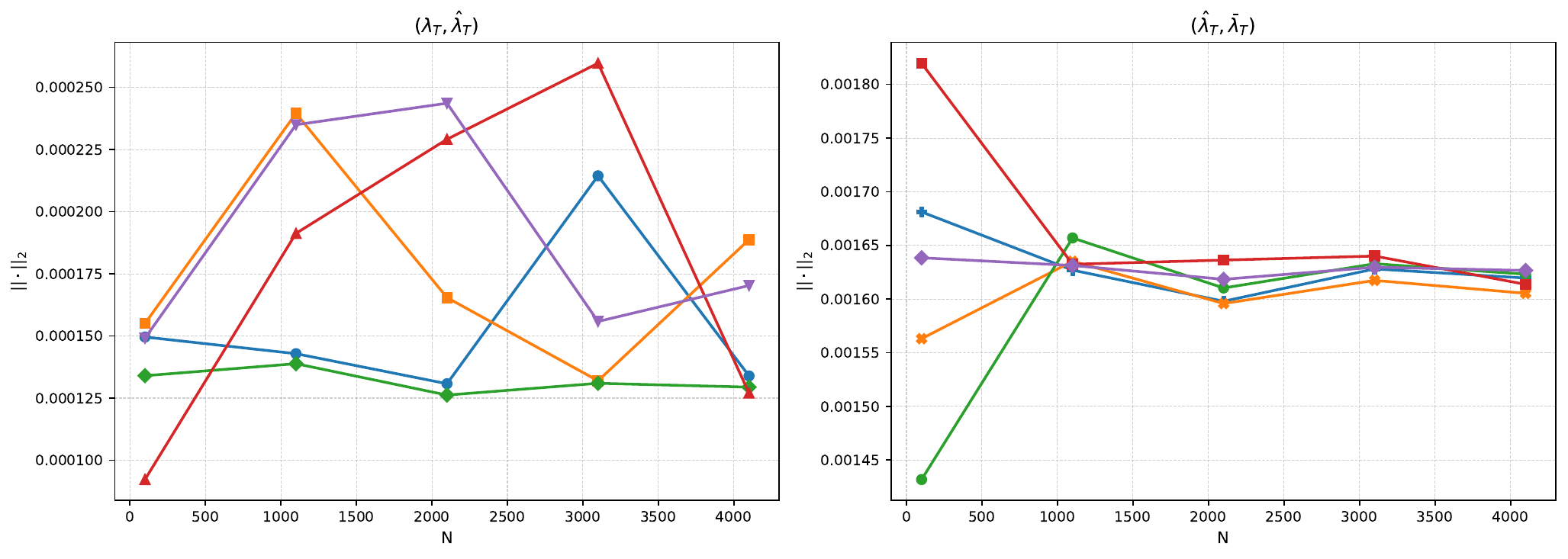}
    \caption{Euclidean norms of intensity differences as a function of the number of nodes \( N \), computed at final time \( T = 100 \), averaged over 100 different runs, using \(\boldsymbol{\alpha} = [0.5, 0.3, 0.2]\) and the assortative SBM set up. Each line corresponds to a different graph realization.}
    \label{fig:N_two_norms}
\end{figure}
Figure \ref{fig:N_inf_norms} demonstrates the intermediary convergence results in $N$ for the true intensity and the auxiliary intensity, and the auxiliary intensity and the mean-field intensity, in Euclidean norm; a necessary stepping stone to the result of Theorem \ref{thm:main_uniform}.
\begin{figure}[h!]
    \centering
    \includegraphics[width=0.9\linewidth]{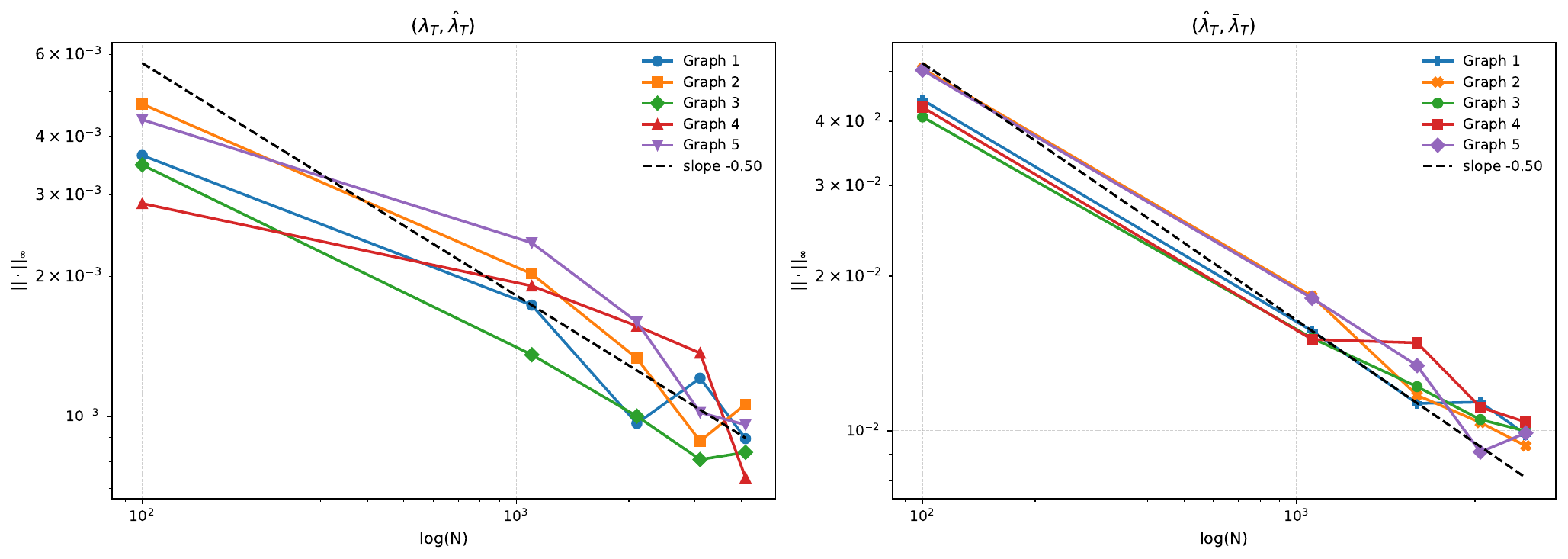}
    \caption{Uniform norms of intensity differences as a function of the number of nodes \( N \), computed at final time \( T = 100 \), averaged over 100 different runs, using \(\boldsymbol{\alpha} = [0.5, 0.3, 0.2]\). Each line corresponds to a different graph realization.}
    \label{fig:N_inf_norms}
\end{figure}

\subsection{Dependence on the time horizon $T$}
Figures \ref{fig:appendix-mean-distances-linear-t} and \ref{fig:appendix-max-distances-linear-t} demonstrate that for fixed $N$, discrepancies between the true and mean-field process remain approximately consistent across the considered time horizons.
 \begin{figure}[h!]
     \centering
     \includegraphics[width=\textwidth]{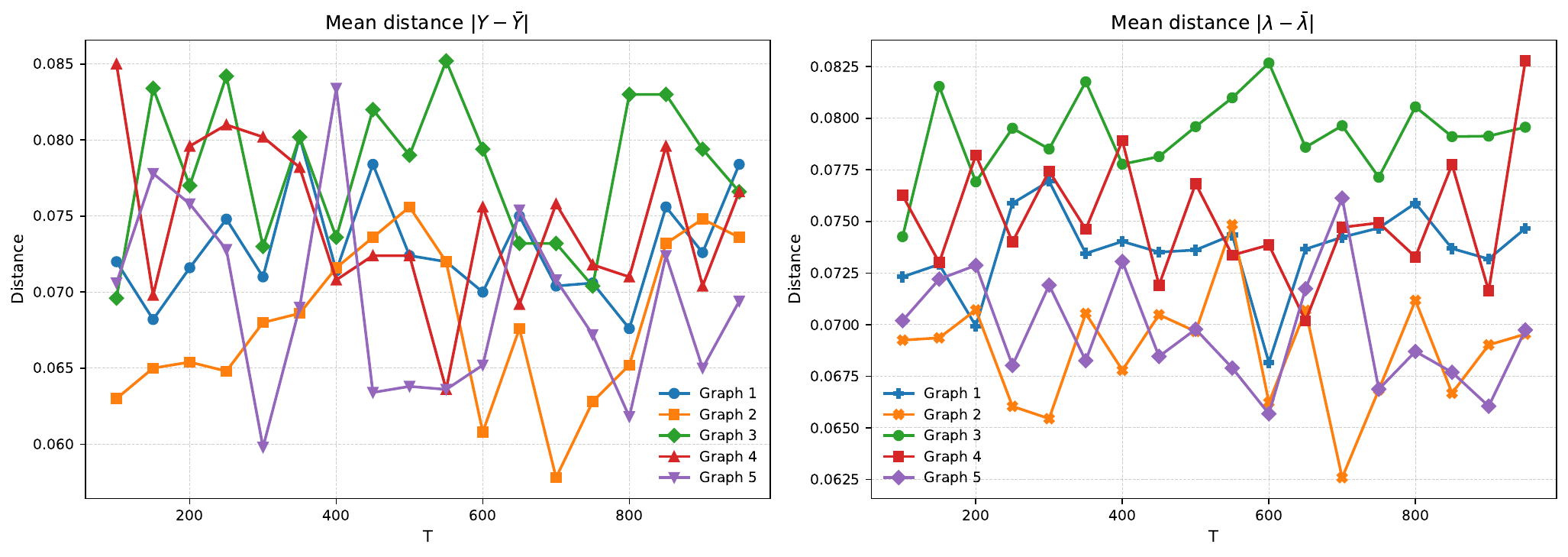}
        \caption{Mean distances as a function of the time horizon $T$, averaged over 100 runs, for five graph realisations of size $N=500$. The left panel shows the mean distance $|Y-\bar{Y}|$ between the true and mean-field processes, and the right panel shows the mean distance $|\lambda-\bar{\lambda}|$.}
     \label{fig:appendix-mean-distances-linear-t}
 \end{figure}

 \begin{figure}[h!]
     \centering
     \includegraphics[width=\textwidth]{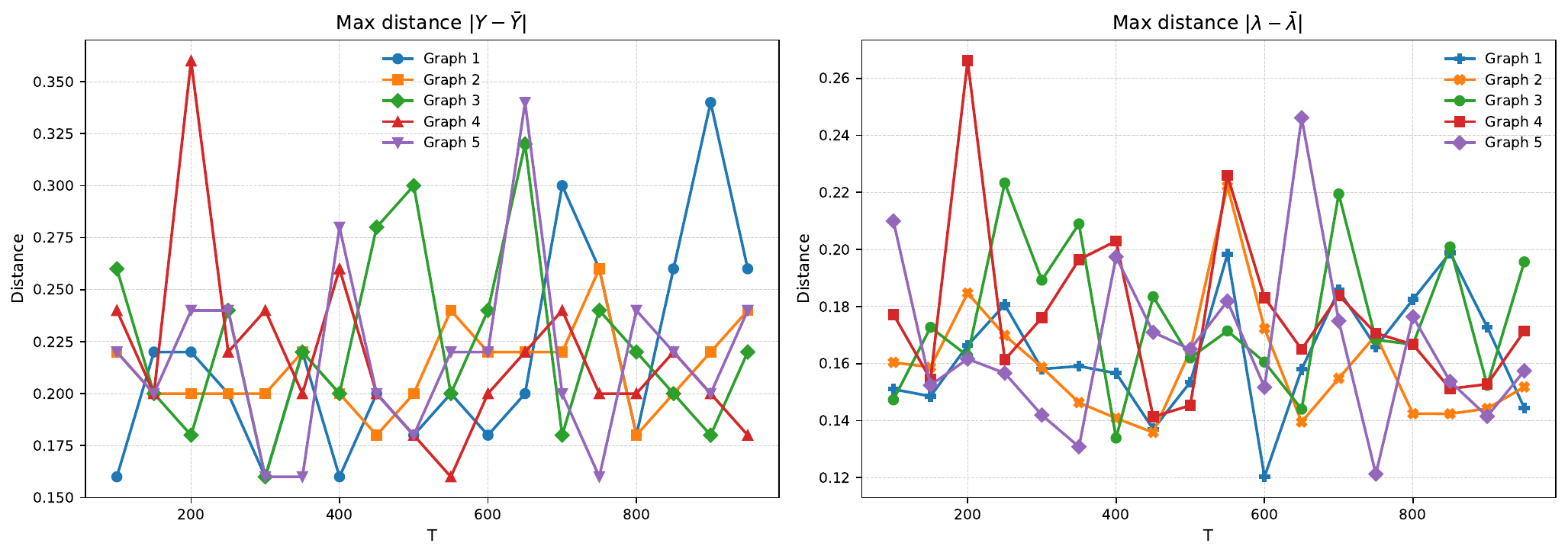}
     \caption{Maximum distances as a function of the time horizon $T$, for five graph realisations of size $N=500$. The left panel shows the max distance $|Y-\bar{Y}|$ between the true and mean-field processes, and the right panel shows the max distance $|\lambda-\bar{\lambda}|$.}
     \label{fig:appendix-max-distances-linear-t}
 \end{figure}

\end{document}